\documentclass[reqno]{amsart}
\usepackage{graphicx,xcolor,cite,mathtools,bbold,bbding}
\usepackage{amssymb,amsmath,amsthm,mathrsfs,paralist,bm,esint,setspace}
\usepackage[ngerman,french,english]{babel}
\RequirePackage[colorlinks,citecolor=blue,urlcolor=blue]{hyperref}
\usepackage{cite}
\usepackage{cases}
\usepackage{tcolorbox}
\usepackage{tikz}
\usepackage{pgfplots}
\usepackage{subcaption}
\usetikzlibrary{arrows.meta}
\pgfplotsset{compat=1.14}
\pgfplotsset{every x tick label/.append style={font=\footnotesize, yshift=0.6ex}}
\pgfplotsset{every y tick label/.append style={font=\footnotesize, xshift=0.5ex}}
\DeclareMathOperator{\Var}{Var}

\DeclareMathOperator{\Cov}{Cov}

\DeclareMathOperator{\fBm}{fBm}

\newcommand{\N}{\mathbb{N}}
\newcommand{\Z}{\mathbb{Z}}

\newcommand{\R}{\mathbb{R}}

\renewcommand{\P}{\mathrm{P}}

\newcommand{\cC}{\mathcal{C}}
\newcommand{\E}{\mathrm{E}}

\newcommand{\1}{\mathbb{1}}
\renewcommand{\d}{{\rm d}}

\newcommand{\e}{{\rm e}}

\renewcommand{\ge}{\geqslant}
\renewcommand{\le}{\leqslant}
\definecolor{CYL}{rgb}{0.3,0.1,0.1}
\definecolor{DK}{rgb}{0.5,0.3,0.5}
\author[D. Khoshnevisan and C.Y. Lee]{Davar Khoshnevisan \and Cheuk Yin Lee}
\address{Department of Mathematics, The University of Utah, Salt Lake City, Utah 84112-0090,
	USA}
\email{davar@math.utah.edu}
\address{School of Science and Engineering, The Chinese University of Hong Kong (Shenzhen), Longgang,
	Shenzhen, Guangdong, 518172, P.R. China}
\email{leecheukyin@cuhk.edu.cn}
\title[Self-similar Gaussian processes on curved boundaries]{On the passage times of self-similar
	Gaussian processes on curved boundaries}
	\thanks{Research supported in part by the National Science Foundation grant DMS-2245242, a research startup fund of the Chinese University of Hong Kong, Shenzhen, and the Shenzhen Peacock fund 2025TC0013.}
\date{Last update: October 30, 2025}
\newtheorem{stat}{Statement}[section]
\newtheorem{proposition}[stat]{Proposition}

\newtheorem{corollary}[stat]{Corollary}
\newtheorem{theorem}[stat]{Theorem}
\newtheorem{lemma}[stat]{Lemma}
\theoremstyle{definition}

\newtheorem{OP}[stat]{Open Problem}
\newtheorem{example}[stat]{Example}

\numberwithin{equation}{section}
\usepackage[cal=dutchcal,calscaled=1.05,
	scr=boondoxo,scrscaled=1.05,
	]{%
mathalfa}%---------------------------------------------------------------------------------------------------------
\date{October 30, 2025}
\begin{document}
\maketitle
%----------------------------------------------------------------------------------------------------
\setcounter{tocdepth}{3}% to get subsubsections in toc
\let\oldtocsection=\tocsection
\let\oldtocsubsection=\tocsubsection
\let\oldtocsubsubsection=\tocsubsubsection
%----------------------------------------------------------------------------------------------------

\renewcommand{\tocsection}[2]{\hspace{0em}\oldtocsection{#1}{#2}}
\renewcommand{\tocsubsection}[2]{\hspace{2.5em}\oldtocsubsection{#1}{#2}}
%\tableofcontents
\begin{abstract} 
	Let $T_{c,\beta}$ denote the smallest $t\ge1$
	that a continuous, self-similar Gaussian process with self-similarity
	index $\alpha>0$ moves at least $\pm  c t^\beta$ units.
	We prove that: (i) If $\beta>\alpha$, then 
	$T_{c,\beta}=\infty$ with positive probability;
	(ii) If $\beta<\alpha$ and $X$ is strongly locally nondeterministic 
	in the sense of Pitt (1978), then $T_{c,\beta}$
	has moments of all order;  and 
	(iii) If $\beta=\alpha$ and $X$ is strongly locally nondeterministic in the sense of Pitt (1978), then there exists a continuous,
	strictly decreasing  function 
	$\lambda:(0\,,\infty)\to(0\,,\infty)$ such that $\mathrm{E}(T_{c,\beta}^\mu)$ is finite
	when $0<\mu<\lambda(c)$ and infinite when $\mu>\lambda(c)$. Together these results
	extend a celebrated theorem of Breiman (1967)
	and Shepp (1967) for passage times
	of a Brownian motion on the critical square-root boundary.
	We briefly discuss two examples: One about fractional Brownian motion, and another
	about a family of linear stochastic partial differential equations.\end{abstract}

\section{Introduction}

Choose and fix a number $c>0$,
let $w=\{w_n\}_{n=1}^\infty$ denote the simple symmetric random walk on $\Z$, and define
$\tau_c  =\inf\{n\in\N:\, |w_n|>c \sqrt n,\ c\sqrt n >1\}$, where
$\inf\varnothing=\infty$. It is not hard to see that
$\P\{\tau_c <\infty\}=1$ for every $c>0$. This follows for example from the law of the iterated
logarithm (LIL). Blackwell and Freedman \cite{BlackwellFreedman} proved  that
$\E(\tau_c)<\infty $ if and only if $c<1.$
Moreover, $w$ can be replaced for example by any random walk that
satisfies $\E(w_1)=0$
and $\Var(w_1)=1$; see Gundy and Siegmund  \cite{GundySiegmund}.

This type of result has been studied in the continuum
as well. Let $W$ denote a standard
1-D Brownian motion, and define
\[
	T_c = \inf\left\{t\ge1:\, |W(t)| > c t^{1/2}\right\}
	\qquad\forall c>0.
\]
By the LIL, $T_c<\infty$ a.s.\ for every $c>0$.  
Breiman \cite{Breiman} and Shepp \cite{Shepp} proved 
at the same time that, for every $c,\mu>0$,
$\E(T_c^\mu)<\infty$ if and only if $c<\mathcal{z}(\mu)$,
where $\mathcal{z}(\mu)$ denotes the smallest positive root
of $x\mapsto M(-\mu\,,1/2\,, x^2/2)$ where $M$ denotes
Kummer's function. The function $\mu\mapsto\mathcal{z}(\mu)$ cannot be evaluated
explicitly but it is known to have a continuous inverse $c\mapsto\mathcal{z}^{-1}(c)$, and
\begin{equation}\label{z(mu)}\begin{split}
	\mathcal{z}(\mu)\sim \mu^{-1/2} \text{ as }\mu\to\infty,\quad&
		\mathcal{z}(\mu) \sim \sqrt{2|\log\mu|}\text{ as }\mu\to0^+,\\
	\mathcal{z}^{-1}(c) \sim c^{-2} \text{ as }c\to0^+,\quad&
		\mathcal{z}^{-1}(c) = \e^{-c^2(1+\mathcal{o}(1))/2}
		\text{ as }c\to\infty.
\end{split}\end{equation}
Moreover, some special evaluations of $\mathcal{z}$ are also known exactly
\cite{Breiman,Shepp}; for instance,
\begin{equation}\label{z(1)z(2)}\textstyle
	\mathcal{z}(1)=1
	\quad\text{and}\quad
	\mathcal{z}(2)= \sqrt{3-\sqrt 6}.
\end{equation}
Novikov \cite{Novikov} presents
exact formulas for closely related passage times. And
Breiman \cite{Breiman} also derived an invariance principle for a large
class of random walks in the domain of attraction of $W$ that include
the simple walk. In particular, Breiman's theory implies that
$\E(\tau_c^\mu)<\infty$ if and only if $c<\mathcal{z}(\mu).$
Breiman's theorem includes the theorem of Black and Freedman 
\cite{BlackwellFreedman} thanks to \eqref{z(1)z(2)}. It also includes
the fact that
$\E(\tau_c^2)<\infty$ if and only if $c<\sqrt{3-\sqrt 6}\approx 0.742,$
established originally by Chow and Teicher \cite{CT}.
See also Teicher and Zhang \cite{TeicherZhang} together with their
detailed bibliography to the previous work on this topic.
Alili and Patie \cite{AliliPatie} discuss some of the latest developments
of passage times of Brownian motion on curved boundaries.

In addition to their consequences in theoretical statistics, passage-time problems 
are also intimately connected to the foundations of 
martingale theory. Notably, Davis \cite{Davis76} found 
the optimal constant in the celebrated Burkholder-Davis-Gundy inequality in part by 
establishing the extremality of the stopping times $\{T_c\}_{c>0}$. Finally, we
mention that such passage
times also play a vital role in the deep
analysis of slow points of the increments of Brownian motion and 
related processes \cite{BassBurdzy96,Davis83,
	DavisPerkins85,GreenwoodPerkinsA,GreenwoodPerkinsB,
	Perkins83,Verzani95}.

In this paper, we aim to study the passage times of self-similar Gaussian processes on
curved boundaries. 
With this in mind, from now on we let $X=\{X(t)\}_{t\ge0}$ denote a centered Gaussian process
with continuous sample functions, and we assume throughout
that $X$ is  \emph{$\alpha$-self similar} for
some number $\alpha>0$ whose value is fixed from now on. Throughout,
$X$ is assumed also to be non degenerate; that is, 
$\Var[X(1)]>0$. It might help
to recall that $\alpha$-self similarity means that the law of $\{r^{-\alpha}X(rt)\}_{t\ge0}$ does not depend on
$r>0$ \cite{Kolmogorov,MandelbrotVanNess}.
This notion was studied in depth by
Lamperti \cite{Lamperti62,Lamperti67,Lamperti72} who alternatively referred to self-similar processes
as ``semi stable.'' For  Markov processes, the subject of self-similarity  has 
been revisited more recently with remarkable advances. The bibliography
has pointers to portions of that literature that are somewhat related to the topic of the present work
\cite{BarczyDoring,BertoinKorchemski,BertoinYor2002A,BertoinYor2002B,CaballeroChaumont,
	CP,CKP,CKPR,DDK,DT,Doring,DoringBarczy,GnedinPitman,GV,KimSongVondracek,KMR,
	LingWang,MicloPatieSarkar,Pardo,PardoRivero,Patie,Rivero2003,Rivero2005,Rivero2007,
	Vidmar,Xiao}.
Self-similar Gaussian processes have also been the subject of many impactful papers a 
representative collection of which can also be found in the references
\cite{CDG,Coeurjolly,Dobrushin,Funaki,GoodmanKuelbs,GHN,HHX,
	HN2017,HN2018,KalbasiMountford,KentWood,MM,MasonXiao,
	MS,Molchan,Muraoka,NX,SM,Song,ZPMPGR}.

Recall  that $X$ is
\emph{strongly locally non-deterministic} (SLND, for short) 
\cite{Berman1978,Berman1987,MonradPitt,CuzickDupreez,Xiao1996,Pitt78} if and only if there exists a number $\mathscr{l}_{_X}>0$ such that
\[
	\Var\left( X(t) \mid X(r);\, 0<r<s\right)
	\ge \mathscr{l}_{_X} (t-s)^{2\alpha},
	\tag{\textsc{slnd}}\label{SLND}
\]
uniformly for all $0<s<t$. Strong local non-determinism is not always
assumed. Whenever we assume that $X$ satisfies \eqref{SLND} 
we explicitly mention that assumption.

The following summarizes what we currently know about the passage times of 
self-similar, continuous Gaussian processes on curved boundaries of
the form $ct^\beta$, and can be viewed as a nontrivial generalization
of the results of Breiman \cite{Breiman} and Shepp \cite{Shepp}, valid
in the absence of both the theories of Markov processes and martingales.

Consider the following stopping times of the Gaussian process $X$:
\[
	T_{c,\beta} = \inf\{ t\ge1:\ |X(t)| > ct^\beta\}
	\qquad\forall c,\beta>0.
\]
The following is the main result of this paper. The super- and the sub-critical 
cases of the next result [parts (1) and (2)] are related
to recent testing methodologies that are proposed in
clinical trials \cite{Lai,LaiETAL,ZhangLai} when
$X$ denotes a fractional Brownian motion.
	
\begin{theorem}\label{th:main}
	\begin{compactenum}[\rm (1)]
	\item If $\beta>\alpha$, then
		$0<\P\{T_{c,\beta}=\infty\}<1$
		for all $c>0$;
	\item If $0<\beta<\alpha$ and \eqref{SLND} holds, then
		$\E\e^{A|\log T_{c,\beta}|^2 } <\infty$
		for all $c>0$ and $0<A<(\alpha-\beta)/(4\beta^2)$.
%		sufficiently small $A=A(c\,,\beta)>0$. 
		Consequently, $\E(T_{c,\beta}^\mu)<\infty$ for all $c,\mu>0$ in this case; 
	\item If $\beta=\alpha$, then for every $c>0$
		there exists $\lambda(c)=\lambda(c\,,X)\ge0$ such that
		$\P\{ T_{c,\alpha}>t\} = t^{-\lambda(c)+o(1)}$ as $t\to\infty$.
		In particular,
		$\E(T_{c,\alpha}^\mu)<\infty$ if $0<\mu<\lambda(c)$, and
		$\E(T_{c,\alpha}^\mu)=\infty$ if $\mu>\lambda(c)$.
		Furthermore, $\lambda$ is non increasing, convex, and 
		continuous on $(0\,,\infty)$, and $\lambda(c)\to0$ as $c\to\infty$; and
	\item Under \eqref{SLND}, $\lambda(c)>0$ for all $c>0$, 
		$\lambda(c) \to \infty$ as $c\to0^+$, and $\lambda$ is strictly decreasing on $(0\,,\infty)$.
	\end{compactenum}
\end{theorem}

We might refer to the function $\lambda$ as the \emph{critical boundary crossing exponent}
of $X$, or ``exponent'' for short.
In the special case that $X$ denotes Brownian motion [$\alpha=1/2$],
$\lambda(c)=\mathcal{z}^{-1}(c)$ for every $c>0$.
We do not know if $\lambda=\mathcal{z}^{-1}$ for other $\frac12$-self similar
Gaussian processes of the type studied here, and can present only partial results in this direction
in \S\ref{sec:comparison}. Nevertheless, 
we are able to offer the following which shows the genericity of 
aspects of \eqref{z(mu)}.

The following is new even for $\alpha=\frac12$ since
Brownian motion is not the only centered, $\frac12$-self similar, continuous
Gaussian process. 

\begin{corollary}\label{cor:main}
	As $c\to\infty$, 
	$c^{-2}\log\lambda(c) \le -(1+\mathcal{o}(1))(2\Var[X(1)])^{-1}$.
	Furthermore: \begin{compactenum}
	\item If there exist $K,\delta>0$ such that
		$\E(|X(t)-X(s)|^2)\le K|t-s|^{2\delta}$ for all $s,t\in[0\,,1]$, then
		$\lambda(c) =\mathcal{O}( c^{-1/\delta})$
		as $c\to0^+$;
	\item Under \eqref{SLND}, 
		$\lambda(c) \gtrsim  c^{-1/\alpha}$ as $c\to0^+$ and
		$c^{-2} \log \lambda(c) \ge -(1+\mathcal{o}(1))(2\mathscr{l}_{_X})^{-1}$
		as $c\to\infty$.
	\end{compactenum}
\end{corollary}

Theorem \ref{th:main} and Corollary \ref{cor:main}
will be proved in \S\ref{proof} below, followed by some comments
on aspects of the theory that might be needed to estimate 
the exponent $\lambda$ by simulation.

\section{Proofs of Theorem \ref{th:main} and Corollary \ref{cor:main}}\label{proof}

We verify Theorem \ref{th:main} in order,
first in the supercritical case [$\beta>\alpha$]. Then, we
prove the remainder of theorem first
in the subcritical case [$\beta<\alpha$], followed by
the critical case $[\beta=\alpha]$, both under
the condition \eqref{SLND}. Corollary \ref{cor:main}
is then proved at the end of this section.

\subsection{The supercritical case} 

We begin by proving part (1) of Theorem \ref{th:main}. Namely, 
\begin{proposition}\label{pr:supercritical}
	$\P\{T_{c,\beta}=\infty\}\in(0\,,1)$ for all $c>0$ and $\beta>\alpha$.
\end{proposition}

The proof hinges on two simple lemmas, the first of which is

\begin{lemma}\label{lem:LIL}
	$\lim_{t\to\infty}\P\left\{ |X(s)| \le  cs^\beta \ \forall s\ge t\right\}=1$
	for every $c>0$ and $\beta>\alpha$.
\end{lemma}

\begin{proof}
	It suffices to prove the following law of the iterated logarithm (hereforth, LIL): With probability one,
	$\limsup_{t\to\infty} |X(t)|/(t^\alpha\sqrt{\log\log t})\le \sqrt{2\Var[X(1)]}.$
	There is a large literature on the law of the iterated logarithm for
	Gaussian processes that includes many variations and improvements
	of the above LIL. For the sake of completeness,
	we present a short modern proof that 
	does not require any additional technical assumptions.

	Let $S(t) = \sup_{s\in[0,t]}|X(s)|$ for all $t>0$.
	According to Talagrand's regularity theorem \cite{Talagrand}, the continuity of $X$
	ensures that $M=\E|S(1)|<\infty$, whence
	$\E|S(t)|=Mt^\alpha$ for all $t>0$ by scaling.
	Thus, it follows from concentration of measure -- more precisely
	Borell's inequality \cite{Borell} -- that, for every
	$t,z>0$,
	\[
		\P\{ |S(t)| \ge Mt^\alpha +z\}
		\le2\exp\left( -\frac{z^2}{2\sup_{s\in[0,t]}\Var[X(s)]}\right)
		= 2\exp\left( -\frac{z^2}{2t^{2\alpha}\Var[X(1)]}\right),
	\]
	thanks to scaling. This proves that for every fixed $c>0$,
	\[
		\P\left\{ |S(t)| \ge t^\alpha\sqrt{c\log\log t}\right\}
		\le\exp\left( - \frac{(c+ \mathcal{o}(1))\log\log t}{2\Var[X(1)]}\right)\qquad
		\text{as $t\to\infty$}.
	\]
	We apply this for any $c>2\Var[X(1)]$ to deduce
	the stated LIL from the Borel-Cantelli lemma.
\end{proof}

The following is the remaining lemma we need for the proof of Proposition
\ref{pr:supercritical}.

\begin{lemma}\label{lem:pos}
	$\P\{ |X(s)| \le c \ \forall s\in[a\,,b]\}>0$ for every $c>0$ and $b>a>0$.
\end{lemma}

\begin{proof}
	Choose and fix $a,b,c>0$ as in the statement of the lemma. Evidently,
	\[
		\Delta := \inf_{s\in[a,b]}\P\{|X(s)|\le c/2\} = 
		\inf_{s\in[a,b]}\P\{|Z|\le c/(2\tau s^\alpha)\} 
		= \P\{|Z| \le c/(2\tau b^\alpha)\}
		>0,
	\]
	where $\tau^2=\Var[X(1)]$,
	and $Z$ has a standard normal distribution. Next, we use continuity
	in order to ensure the existence of a number $n\in\N$
	such that
	\[\textstyle
		\mathscr{E}:=\P\left\{ \sup_{s,t\in[a,b]:\ |t-s|\le 1/n}
		|X(t) - X(s)| \ge c/2 \right\} \le \Delta/2.
	\]
	Define $t_i = a + i(b-a)n^{-1}$ for all $i=0,\ldots,n$, and observe that
	\begin{align*}
		\Delta \le \P\left\{ |X(t_i)|\le c/2 \right\}
		&\le \P\left\{ |X(s)|\le c \ \forall 
		s\in[t_i\,,t_{i+1}]\right\} + \mathscr{E}\\
		&\le \P\left\{ |X(s)|\le c \ \forall 
		s\in[t_i\,,t_{i+1}]\right\} + \frac{\Delta}{2},
	\end{align*}
	for all $i=1,\ldots,n-1$, whence
	$\min_{0\le i\le n-1}\P\{ |X(s)|\le c \ \forall  s\in[t_i\,,t_{i+1}]\}
	\ge \Delta/2.$
	Consequently, the Gaussian correlation inequality \cite{Royen,LatalaMatlak} 
	implies that $\P\{ |X(s)|\le c\ \forall s\in[a\,,b]\} $ is greater than or equal to
	$\prod_{i=0}^{n-1}\P\{ |X(s)|\le c \ \forall  s\in[t_i\,,t_{i+1}]\}\ge (\Delta/2)^n>0.$
	This completes the proof.
\end{proof}
 
 We are ready to demonstrate Proposition \ref{pr:supercritical}, thereby
 conclude the section.
 
 \begin{proof}[Proof of Proposition \ref{pr:supercritical}]
 	By the Gaussian correlation inequality
	\cite{Royen,LatalaMatlak}, for all $\beta>\alpha$, $m\ge1$, $c>0$,
 	\begin{align*}
	 	\P\{T_{c,\beta} =\infty\}	&=\P\left\{ |X(s)| \le cs^\beta\
			\forall s\ge1\right\}\\
			&\ge \P\left\{ |X(s)| \le cs^\beta\ \forall s\in[1\,,m]\right\}\cdot
			\P\left\{ |X(s)| \le c s^\beta\ \forall s\ge m\right\}\\
		&\ge \P\left\{ |X(s)| \le c\ \forall s\in[1\,,m]\right\}\cdot
			\P\left\{ |X(s)| \le c s^\beta\ \forall s\ge m\right\}.
	\end{align*}
	Choose and fix a sufficiently large $m\gg1$ in order to deduce
	from Lemma \ref{lem:LIL} that
	$\P\{ |X(s)| \le c s^\beta\ \forall s\ge m\}\ge\frac12.$
	These observations and Lemma \ref{lem:pos} together prove that
	$\P\{T_{c,\beta}=\infty\}>0$. It therefore remains to prove
	that $\P\{T_{c,\beta}=\infty\}<1$, but that follows easily
	from the fact that
	$\P\{T_{c,\beta}<\infty\}\ge \P\{|X(2)|>c2^\beta\}$ is strictly 
	positive since $X(2)$ has a non-degenerate centered normal distribution.
	This concludes the proof.
 \end{proof}
 
\subsection{The subcritical case}

The following implies part (2) of Theorem \ref{th:main}.

\begin{proposition}\label{pr:subcritical}
	If $c>0$, $\beta\in(0\,,\alpha)$ and
	\eqref{SLND} holds, then $\E\e^{A|\log T_{c,\beta}|^2}<\infty$ for all  
	$0<A<\frac{\alpha-\beta}{4\beta^2}$.
\end{proposition}

\begin{proof}
	We aim to prove that for all $c>0$ and $\beta\in(0\,,\alpha)$,
	\begin{equation}\label{tail:goal}
		\P\{T_{c,\beta} > t\}
		\le \exp\left\{-\left(\frac{\alpha-\beta}{4\beta^2}\right)(\log t)^2
		+\mathcal{O}(\log t)\right\}\qquad\text{as $t\to\infty$}.
	\end{equation}
	It is not hard to see that this yields the result.
	
	Let $Y(t)=\e^{-t\alpha}X(\e^t)$ for all $t\in\R$.
	Let $\mathcal{X}(t)$
	denote the $\sigma$-algebra generated by $X(s)$ for $s\in[0\,,\e^t]$.
	Then, the process $Y$ is adapted to the filtration 
	$\{\mathcal{X}(\e^t)\}_{t\ge0}$, and for every integer $j\ge1$,
	\begin{align*}
		&\P\left\{ |X(\e^j)| \le c\,\e^{j\beta} \mid \mathcal{X}(\e^{j-1})\right\} = 
			\P\{ |Y(j)|\le c\,\e^{-(\alpha-\beta)j} \mid \mathcal{X}(\e^{j-1}) \}\\
		&=\frac{1}{\sqrt{2\pi\Var[Y(j)\mid \mathcal{X}(\e^{j-1})]}}
			\int_{-c\exp\{-(\alpha-\beta)j\}}^{c\exp\{-(\alpha-\beta)j\}}\exp\left(
			-\frac{\left| x-\E[Y(j)\mid\mathcal{X}(\e^{j-1})] \right|^2}{
			\Var[Y(j)\mid\mathcal{X}(\e^{j-1})]}\right)\d x\\
		&\le \frac{2c \exp\{-(\alpha-\beta)j\}}{\sqrt{2\pi\Var[Y(j)\mid \mathcal{X}(\e^{j-1})]}}.
	\end{align*}
	Thanks to \eqref{SLND},
	\begin{align*}
		\Var[Y(j)\mid\mathcal{X}(\e^{j-1})] &= \e^{-2j\alpha}
		\Var[X(\e^j)\mid\mathcal{X}(\e^{j-1})]\\
		&\ge\mathscr{l}_{_X}\e^{-2j\alpha}\left(\e^j-\e^{j-1}\right)^{2\alpha}
		=\mathscr{l}_{_X}(1-\e^{-1})^{2\alpha}.
	\end{align*}
	Therefore, for every integer $j\ge1$,
	\[
		\P\left\{ |X(\e^j)| \le c\,\e^{j\beta} \mid \mathcal{X}(\e^{j-1})\right\}
		\le \frac{c\sqrt{2/(\mathscr{l}_{_X}\pi)}}{(1-\e^{-1})^\alpha}\,
		\e^{-(\alpha-\beta)j},
	\]
	almost surely. We now apply successive conditioning in order
	to find that
	\begin{align*}
		&\P\{T_{c,\beta}>\e^{m\beta}\}%=\P\{|X(s)| \le cs^\beta\ \forall s\le\e^{m\beta}\} 
			\le \P\left\{ |X(\e^j)|\le c\,\e^{j\beta}\ \forall 1\le j\le m\right\}\\
		&\le \prod_{j=1}^m \left( \frac{c\sqrt{2/(\mathscr{l}_{_X}\pi)}}{(1-\e^{-1})^\alpha}\,
			\e^{-(\alpha-\beta)j}\right)
%			= \left( \frac{\sqrt{2/(\mathscr{l}_{_X}\pi)}}{(1-\e^{-1})^\alpha}\right)^m
%			\e^{-(\alpha-\beta)\sum_{j=1}^m j}\\
%		&
		= \left( \frac{c\sqrt{2/(\mathscr{l}_{_X}\pi)}}{(1-\e^{-1})^\alpha}\right)^m
			\e^{-(\alpha-\beta)m(m+1)/2}.
	\end{align*}
	Set $K = 2\log_+ ( c\sqrt{2/(\mathscr{l}_{_X}\pi)} (1-\e^{-1})^{-\alpha}).$
	If $n\ge\e^\beta$ is an integer then there exists a unique integer $m\ge1$ such that
	$\e^{m\beta} \le n\le \e^{(m+1)\beta},$
	whence, as $n\to\infty$,
	\begin{align*}
		&\P\{T_{c,\beta}>n\} \le \P\{T_{c,\beta}>\e^{m\beta}\}\\
		&\le  \exp\left( -\frac{(\alpha-\beta)m(m+1)-Km}{2}\right)
		= \exp\left( -\frac{(\alpha-\beta)(\log n)^2}{2\beta^2} 
		+ \mathcal{O}(\log n)\right).
	\end{align*}
	This yields \eqref{tail:goal} as $t\to\infty$ along integers.
	Since $\log t \sim \log\lfloor t\rfloor$ as $t\to\infty$, 
	\eqref{tail:goal} holds as $t\to\infty$ along reals as well.
	This completes the proof.
\end{proof}

 \subsection{The critical case} 
 We have already proved parts (1) and (2) of Theorem
 \ref{th:main}. The remaining assertions of Theorem \ref{th:main}
 are included in the following.

\begin{proposition}\label{pr:X}
	Even without  \eqref{SLND},
	\[\textstyle
		\lambda(c) = -\lim\limits_{t\to\infty} (\log t)^{-1}\log
		\P\left\{{ \sup_{s\in[b,t]} ( |X(s)|/s^\alpha)} \le c\right\}
	\]
	exists for every $b,c>0$, is finite, and does not depend on $b$. 
	Moreover, $\lambda(c)$ satisfies $\lim_{c\to\infty} \lambda(c)=0$, and $\lambda$ is a 
	non increasing, convex, continuous function on $(0\,,\infty)$.
	Under \eqref{SLND}, $\lambda(c)>0$
	for all $c>0$, $\lim_{c\downarrow0} \lambda(c)=\infty$, and $\lambda$ is strictly decreasing.
\end{proposition}

\begin{proof}
	First of all we may note that $X(t)\to0$ in probability as $t\downarrow0$ 
	because of scaling. Therefore the continuity of $X$ ensures that $X(0)=0$.
	Now we follow Lamperti \cite{Lamperti62}
	and define a process $Y=\{Y(t)\}_{t\in[-\infty,\infty)}$ as follows:
	\begin{equation}\label{Y}
		Y(-\infty)=0\quad\text{and}\quad
		Y(t) = \e^{-t\alpha} X(\e^t)\qquad\forall t\in\R.
	\end{equation}
	The process $Y$ is the Ornstein-Uhlenbeck process associated to $X$,
	and made a brief appearance earlier in the proof of 
	Proposition \ref{pr:subcritical}.
	A direct computation yields
	$\E[Y(t)Y(s)]  = \e^{-\alpha |t-s|}\E[X(\e^{|t-s|})X(1)]$
	for all $s,t\in\R.$
	Therefore, it follows that $\{Y(r)\}_{r\in(-\infty,\infty)}$ is a stationary,
	time-reversible Gaussian process, and that $Y$ has continuous sample functions
	on $[-\infty\,,\infty)$. Now,
	\begin{align*}\textstyle
		\P\left\{\sup_{s\in[b,t]} \left( |X(s)|/s^\alpha\right) \le c\right\}
			&\textstyle=\P\left\{\sup_{r\in[\log b,\log t]}|Y(r)|\le c\right\}\\
		&\textstyle=\P\left\{\sup_{r\in[0,\log(t/b)]}|Y(r)|\le c\right\},
	\end{align*}
	for all $t>b>0$ and $c>0$.
	Choose and fix a number $c>0$ and define
	$f(t) = \P\{ |Y(r)|\le c\ \forall r\in[0\,,t]\}$ $[t\ge 0].$
	Thanks to the Gaussian correlation inequality \cite{Royen,LatalaMatlak}, 
	$f$ satisfies the following
	for all $s,t>0$:
	\begin{align*}\textstyle
		f(t+s) 
%			&\textstyle= \P\left\{\sup_{r\in[0,t]}|Y(r)|\le c \,, 
%			\sup_{r\in[t,t+s]}|Y(r)|\le c\right\}\\
%		&\textstyle\ge \P\left\{\sup_{r\in[0,t]}|Y(r)|\le c\right\}
%			\P\left\{\sup_{r\in[t,t+s]}|Y(r)|\le c\right\}
%			\\&\textstyle
		\ge \P\left\{\sup_{r\in[0,t]}|Y(r)|\le c\right\}
		\P\left\{\sup_{r\in[t,t+s]}|Y(r)|\le c\right\} = f(t)f(s),
	\end{align*}
	where we have used stationarity in the last identity. 
	That is, $f$ is supermultiplicative.  
	
	Next, let us suppose that
	$f(t)=0$ for some $t>0$. Because $f$ is supermultiplicative, 
	$f(t)\ge [f(t/2)]^2$ and hence $f(t/2)=0$. Repeat this inductively to see that
	$f(t/2^n)=0$ for every integer $n\ge0$. It follows from the continuity of $Y$ that
	$0 = \lim_{n\to\infty} f(t/2^n) = \P\{|Y(0)|\le c\}.$
	This cannot be since $Y(0)$ has a non-degenerate centered normal
	distribution. Therefore, it follows that $f(t)>0$ for all $t>0$. That and the supermultiplicative
	property of $f$ together show that $-\log f$ is subadditive, whence
	$\lambda(c) = -\lim_{t\to\infty} t^{-1}\log f(t) = -\sup_{t\ge1} t^{-1}\log f(t),$
	thanks to the Fekete lemma, generally ascribed to \cite{Fekete}. We can also see that
	$\lambda(c) \le -\log f(1) <\infty$. Moreover,
	\begin{equation}\label{lambda:UB}\textstyle
		\lambda(c) \le -\log\P\left\{\sup_{r\in[0,1]}|Y(r)|\le c\right\}\downarrow0,
	\end{equation}
	as $c\uparrow\infty$, since $Y$ is continuous hence locally bounded.
	
	By the log-concavity of Gaussian measures 
	(see Borell \cite[Corollary 2.1]{Borell74}), for every $s \in (0\,,1)$, $a,b>0$, 
	and $n \in \N$,
	\begin{align*}
		&\textstyle \log\P\left\{ \sup_{r\in[0,n]}|Y(r)| \le sa + (1-s)b \right\}\\
		&\hskip0.5in\textstyle \ge s\log\P\left\{ \sup_{r\in[0,n]}|Y(r)| \le a \right\}
		+ (1-s)\log\P\left\{ \sup_{r\in[0,n]}|Y(r)| \le b \right\}.
	\end{align*}
	Divide both sides by $-n$ and then let $n\to\infty$ in order to see that
	$\lambda(sa+(1-s)b) \le s\lambda(a)+(1-s)\lambda(b).$
	In other words, $\lambda$ is convex, 
	whence continuous, on $(0\,,\infty)$. The asserted monotonicity of $\lambda$ is 
	manifest. 
	
	Finally, suppose that $X$ satisfies \eqref{SLND}. Recall the stationary process $Y$
	in \eqref{Y}, and  define $\mathcal{Y}(t)$ to be the $\sigma$-algebra generated
	by $\{Y(s)\}_{s\in(-\infty,t]}$. Thanks to \eqref{SLND},
	for all $n\in\Z_+$,
	\begin{align}\begin{split}\label{Var(Y|Y)}
		\Var[Y(n)\mid \mathcal{Y}(n-1)] 
		&= \e^{-2n\alpha}
		\Var\left[ X(\e^{n})\mid \mathcal{X}(\e^{n-1})\right]\\
		&\ge \mathscr{l}_{_X}\e^{-2n\alpha}(\e^{n}-\e^{n-1})^{2\alpha}
		= \mathscr{l}_{_X}(1-\e^{-1})^{2\alpha},
	\end{split}\end{align}
	where first identity hinges on the already-proved fact that $X(0)=0$.
	Since conditionally Gaussian laws are themselves Gaussian,
	\begin{align*}
		&\P\left\{ |Y(n)|\le c\mid\mathcal{Y}(n-1)\right\}\\
		&= \frac{1}{\sqrt{2\pi \Var[Y(n)\mid \mathcal{Y}(n-1)]}}
			\int_{- c }^ c  \exp\left( 
			-\frac{|x-\E(Y(n)\mid\mathcal{Y}(n-1))|^2}{2
			\Var[Y(n)\mid \mathcal{Y}(n-1)]}\right)\d x.
	\end{align*}
	One-dimensional Gaussian laws are unimodal, as can be checked
	directly, or deduced from the much more general 
	Anderson's shifted-ball inequality \cite{Anderson55}.
	This implies that
	\begin{align*}
		&\P\left\{ |Y(n)|\le c\mid\mathcal{Y}(n-1)\right\}\\
		&\le \frac{1}{\sqrt{2\pi \Var[Y(n)\mid \mathcal{Y}(n-1)]}}
			\int_{- c }^ c  \exp\left( 
			-\frac{|x|^2}{2
			\Var[Y(n)\mid \mathcal{Y}(n-1)]}\right)\d x\\
		&=\P\left\{ |Z| \le \frac{c}{\sqrt{\Var[Y(n)\mid \mathcal{Y}(n-1)]}} \right\}
		\le \P\left\{|Z| \le \frac{c}{\sqrt{\mathscr{l}_{_X}}{(1-\e^{-1})}^{\alpha}}\right\},
	\end{align*}
	where $Z$ denotes a standard normal random variable, and the 
	last inequality is obtained by applying \eqref{Var(Y|Y)}.
	In particular, $\P\{\sup_{s\in[1,\exp(n)]} \left( |X(s)|/s^\alpha\right) \le c\}$
	is equal to
	\begin{align}\begin{split}\label{X:bc:UB}
		\P\left\{ \sup_{r\in[0,n]}|Y(r)|\le c\right\}
		&\le \P\left\{ \max_{1\le i\le n}|Y(i)| \le c\right\}\\
		&\le \P\left\{|Z| \le 
		\frac{c}{\sqrt{\mathscr{l}_{_X}}{(1-\e^{-1})}^{\alpha}}\right\}^n,
	\end{split}\end{align}
	thanks to successive conditioning. 
	Take the logarithm of both sides, divide both
	sides by $-n$, and let $n\to\infty$ in order to see that
	\begin{equation}\label{lambda:LB}
		\lambda(c) \ge -\log\P\left\{|Z| \le 
			\frac{c}{\sqrt{\mathscr{l}_{_X}}{(1-\e^{-1})}^{\alpha}}\right\},
	\end{equation}
	which is $>0$ for all $c>0$
	and tends to infinity when $c\downarrow0$. 
	
	Since $\lambda$ is strictly
	positive and tends to zero, its convexity readily also implies its strict
	decrease under condition \eqref{SLND}. In fact, if $\lambda(a) = \lambda(b)$ for some $b>a>0$, 
	then the convexity and monotonicity of $\lambda$
	on $(a\,,\infty)$ together imply that $\lambda$ is constant on 
	$[b\,,\infty)$ and  since $\lim_{c\to\infty}\lambda(c)=0$, this constant 
	must be 0 and thus $\lambda(b) = 0$. We proved in \eqref{lambda:LB} that 
	\eqref{SLND} implies that $\lambda>0$ 
	everywhere. Therefore, we are led to a contradiction.
	This implies the strict decrease of $\lambda$ under \eqref{SLND},
	and completes the proof of Proposition \ref{pr:X}.
\end{proof}

\subsection{Proof of Corollary \ref{cor:main}}

First we establish the asserted behavior of $\lambda$
near infinity. Let $Y$ be the process that was defined in \eqref{Y}.

Accoding to Fernique's theorem \cite{Fernique}, the continuity of $Y$ ensures
that $C=\E\sup_{r\in[0,1]}|Y(r)|$ is finite. Therefore,
Borell's inequality \cite{Borell}, and the fact that $\Var[Y(r)]=\Var[X(1)]$
for all $r>0$, together yield
$c^{-2} \log \P\{\textstyle\sup_{r\in[0,1]}|Y(r)| > c\} \le 
-(1+\mathcal{o}(1))(2\Var[X(1)])^{-1}$ as $c\to\infty$.
This and \eqref{lambda:UB} together imply the first assertion of Corollary \ref{cor:main}.
In order to establish part 1 of the corollary, we may assume
that $\E(|X(t)-X(s)|^2)\le K|t-s|^{2\delta}$ for all $s,t\ge0$.
Now, for every $1\ge t\ge s\ge 0$, $\E(|Y(t)-Y(s)|^2)$ is equal to
\begin{align*}
	&\E\left(\left| \frac{X(\e^t)}{\e^{\alpha t}} - 
		\frac{X(\e^s)}{\e^{\alpha s}}\right|^2 \right)\\
		&\le 2\e^{-2\alpha t}\E\left(\left| X(\e^t)-X(\e^s)\right|^2\right)
		+2 \left( 1 - \e^{-\alpha (t-s)}\right)^2\E(|Y(s)|^2)\\
	&\le 2K\e^{2(\delta-\alpha) t}\left(1-\e^{-(t-s)}\right)^{2\delta}
		+2 \alpha^2(t-s)^2\E(|Y(0)|^2)
		\le L|t-s|^{2\delta},
\end{align*}
where $L=L(K\,,\alpha\,,\delta)>0$. We have used the fact that,
thanks to the hypothesis of this part and scaling,
$\E(|X(t)-X(s)|^2)\le K|t-s|^{2\delta}$ for all $s,t\in[0\,,\e]$.
Since $X$ -- hence $Y$ -- is assumed to be non degenerate, this forces $\delta$ to be
in $(0\,,1]$. With this observation in mind, we may define
a distance $d$ on $\R$ by setting
$d(s\,,t) = L|t-s|^\delta$ for all $s,t\in\R$.
Let $N$ denote the metric entropy of the metric space
$([0\,,1]\,,d)$. That is,
for every $\varepsilon\in(0\,,1)$, let $N(\varepsilon)$ denote the
minimum number of $d$-balls of radius $\varepsilon$ needed
to cover $[0\,,1]$. The form of the distance function $d$
readily implies that $N(\varepsilon) \lesssim \varepsilon^{-1/\delta}$
uniformly for all $\varepsilon\in(0\,,1)$. Therefore, we may apply the Gaussian correlation inequality \cite{Royen,LatalaMatlak} and Lemma 2.2 of
Talagrand \cite{T95} -- see also Dalang, Lee, Mueller, and Xiao \cite[Lemma 3.4]{DLMX21}
for a more precise statement -- in order to see that there exists $M=M(K\,,\alpha\,,\delta)>0$
such that for all $c \in (0\,,1)$,
\begin{align*}\textstyle
	\P\left\{ \sup_{s\in[0,1]}|Y(s)|\le c \right\}
	& \textstyle \ge \P\left\{ |Y(0)| \le c/2\right\} \cdot 
		\P\left\{ \sup_{t,s\in[0,1]}|Y(t)-Y(s)|\le c/2 \right\}\\
	& \ge \P\left\{ |Y(0)| \le c/2\right\} \cdot M^{-1}\exp(- M/c^{1/\delta}).
\end{align*}
Since $\Var[Y(0)] = \Var[X(1)] > 0$, the preceding together with \eqref{lambda:UB} implies that $\lambda(c) \le M c^{-1/\delta} (1+\mathcal{o}(1))$ as $c \to 0^+$. This proves part 1 of the corollary.

If $Z$ follows a standard normal distribution, then 
$\P\{\textstyle |Z| \le c\, \mathscr{l}_{_X}^{-1/2}\}\ge  1 - 2\e^{-c^2/(2\mathscr{l}_{_X})}$
for all $c>0.$ Thus, we can conclude part 2 of the corollary
since \eqref{lambda:LB} is valid when $X$ satisfies \eqref{SLND}.

It remains to show the lower bound for $\lambda(c)$ as $c\to0^+$. 
For each $c \in (0\,,1)$, similarly to the way \eqref{X:bc:UB} was proved, we may take $m = \lfloor c^{-1/\alpha}n \rfloor$ and deduce that $\P\{\sup_{s\in [1,\exp(n)]} ( |X(s)|/s^\alpha) \le c \}$
is equal to
\begin{align*}\textstyle
	\P\left\{ \sup_{r \in [0,n]} |Y(r)| \le c \right\}
	&\le \P\left\{ \max_{1\le i \le m} |Y(c^{1/\alpha} i)| \le c \right\}\\
	&\le \P\left\{ |Z| \le  c\mathscr{l}_{_X}^{-1/2} (1-\exp(-c^{1/\alpha}))^{-\alpha} \right\}^m.
\end{align*}
Then, take the logarithms, divide by $-n$, and let $n\to\infty$ to see that
$\lambda(c)$ is bounded from below by $-c^{-1/\alpha} \log \P \{ |Z| \le 
 c\mathscr{l}_{_X}^{-1/2} (1-\exp(-c^{1/\alpha}))^{-\alpha}  \}.$
This concludes the proof of Corollary \ref{cor:main} 
since $c(1-\exp(-c^{1/\alpha}))^{-\alpha} \to 1$ as $c \to 0^+$.\qed

\section{Examples}
In this section we present a few families of self-similar Gaussian processes
to which all of the assertions of Theorem \ref{th:main} and Corollary \ref{cor:main}
apply. In the next section we discuss a number of relations between the two examples
of this section.

\subsection{Fractional Brownian motion}
 
Let us choose and fix a number
$H\in(0\,,1)$ and let $B=\{B(t)\}_{t\ge0}$ denote a fractional Brownian motion 
(fBm, for short) with index
$H$; that is, $B$ is a centered Gaussian process with
\begin{equation}\label{Cov}
	\Cov[ B(t) \,, B(s)] = \tfrac12\left[ t^{2H} + s^{2H} - |t-s|^{2H}\right]\qquad\forall s,t\ge0.
\end{equation}
As an immediate consequence of \eqref{Cov} we can see that $B$ is $H$-self similar, 
and $\E(|B(t)-B(s)|^2) = |t-s|^{2H}$ for all $s,t\ge0$.
Moreover, Pitt \cite{Pitt78} has proved that fBm satisfies \eqref{SLND} with $\alpha=H$.
As such, all of the findings
of Theorem \ref{th:main} and Corollary \ref{cor:main} apply to $B$.

One might think of $B(t)$ as the position of a one-dimensional particle
at time $t$ when the mean-squared displacement is
$\E(|B(t)|^2)^{1/2} = |t|^H$ for 
every $t\ge0$. 
When $H=\frac12$ the process $B$ is a classical model for diffusion,
and it can serve as a model for 1-D anomalous diffusion when $H\neq\frac12$
\cite{Mandelbrot}. A large body of applied work
defines fBm  as the solution to the Langevin equation $\d B(t) = \dot\eta(t)\,\d t$ where  $\dot\eta$
is a centered Gaussian noise with 
$\Cov[\dot\eta(t)\,,\dot\eta(s)]=\text{const}\cdot |t-s|^{-H-2}.$
Of course this type of exposition has to be interpreted
with some care since $t\mapsto |t|^{-H-2}$ is not locally integrable
and hence cannot be a covariance functional in the traditional sense. 
But certainly $B$ has continuous sample functions  \cite{Kolmogorov,MandelbrotVanNess} and so
$\dot\eta = \d B(t)/\d t$ is a well-defined  generalized
Gaussian random field \cite{Dudley,Minlos,DFL,Gross}.
Lysy et al \cite{LPHFMVM} 
present a rigorous comparative analysis of fBm and other Gaussian models of 
anomalous diffusion in biological fluids, with additional pointers to the 
physics and biology literature on this topic; see also Verdier et al \cite{VLCVM}
and their extensive bibliography.

\subsection{Linear SPDEs}\label{subsec:SPDE}
Mueller and Wu \cite{MWB} describe a connection between the
increments of a fBm of index $H$ (the ``pinned string'') and a family of 
stochastic partial differential equations (SPDEs, for short). 
Their work builds on earlier linearization ideas \cite{LeiNualart,Walsh,KSXZ,HP,FKM}
in the SPDE literature.
Here, we present an extension, as well as a refinement, of these results.

Throughout this discussion we choose and fix some $d\in\N$ and, 
for every real number $\gamma>0$, let 
$\mathcal{L}_\gamma = -(-\Delta)^{\gamma/2}$
denote the fractional Laplace operator
of index $\gamma/2$, acting on real-valued functions on $\R^d$. Recall 
that $\mathcal{L}_\gamma$ is a pseudo-differential operator
from symbolic calculus, defined via 
$\hat{\mathcal{L}}_\gamma (\xi) = -\|\xi \|^\gamma$ for all $\xi\in\R^d.$
It follows that the fundamental solution to the
corresponding heat operator $\partial_t - \mathcal{L}_\gamma$ 
is of the form $G(s\,,y\,;t\,,x)=G(t-s\,,x-y)$,
where
\begin{equation}\label{hat:G}
	\hat{G}(t\,,\xi) = \exp(-t\| \xi \|^\gamma)\qquad\forall t>0,\ \xi\in\R^d,
\end{equation}
and $\hat{\phantom{a}}$ denotes the Fourier transform in the spatial variable $x$.

Let us collect some elementary facts about the heat kernel $G$. The following
result is a well-known folklore fact, but we were not able to find a complete proof of
these particular statements explicitly, particularly that of part (5), 
and so we will supply the short proof
for the sake of completeness. As is customary, let $\mathcal{C}_0^\infty(\R^d)$
denote the collection of all infinitely differentiable functions on $\R^d$
that vanish at infinity.

\begin{lemma}\label{lem:G}
	\begin{compactenum}[\rm (1)]
	\item $G(t)\in \mathcal{C}_0^\infty(\R^d)\cap L^2(\R^d)$ for all $t>0$;
	\item If $0<\gamma \le 2$ then $G(t)>0$ is a probability density function on $\R^d$ for all $t>0$;
	\item $G(t)\in L^1(\R^d)$ and 
		$\int_{\R^d} G(t\,,x)\,\d x=1$ for all $t>0$;
	\item $G(t+s)=G(t)*G(s)$ for all $t,s>0$;
	\item If $\gamma>2$, then $G(t)$ is signed for all $t>0$.
	\end{compactenum}
\end{lemma}

\begin{proof}
	First of all, $G(t)\in L^2(\R^d)$ for all $t>0$ since
	$\hat{G}(t)\in L^2(\R^d)$ [Plancherel's theorem].
	Now, for every multi-index ${\bf a}\in \N^d$,
	\eqref{hat:G} and the inversion formula for Fourier transforms together yield
	\begin{equation}\label{G}
		\partial^{{\bf a}}_x G(t\,,x)= \frac{(-i)^{|{\bf a}|}}{(2\pi)^d}\int_{\R^d}
		\e^{-ix\cdot \xi - t\| \xi \|^\gamma}\textstyle\prod_{j=1}^d\xi_j^{a_j}\,\d\xi
	\end{equation}
	as a version of the $|{\bf a}|$-th spatial derivative of $G$, where $\partial^0_x G=G$. 
	Since the preceding integral converges absolutely, the Riemann-Lebesgue lemma implies part (1). Part (2)
	is basically saying that $G$ denotes the transition probabilities of an isotopic
	$\gamma$-stable process on $\R^d$, itself a consequence of the L\'evy-Khintchine
	formula \cite{Bertoin,Kyprianou,Sato}, and the strict positivity of
	$G(t)$ is a well-known consequence of subordination \cite{Sato}. 
	This reduces the proof of (3)
	to the case that $\gamma>2$. With that case in mind, 
	consider the differential operator,
	$D = \Delta^d = (\sum_{j=1}^{d}\partial^{2}_{\xi_j})^d.$
	The inverse Fourier multiplier of $D$ is given by $\check{D}(x) = i^{2d}\|x\|^{2d}$.
	Therefore, the inversion formula yields
	$\sup_{x\in\R^d} \|x\|^{2d}  |G(t\,,x)| 
	\lesssim \sum_{j=1}^d \int_{\R^d} | D\hat{G}(t\,,\xi)|\,\d\xi,$
	which can be seen directly to be finite since $\gamma>2$.
	With this under way, and since we have seen that $G(t)$ is bounded for every $t>0$
	[part (1)],  it follows that $|G(t\,,x)| \lesssim (1+\|x\|)^{-2d},$
	uniformly for
	all $x\in\R$, whence follows that $G(t)\in L^1(\R^d)$ for every $t>0$.
	Because $\int_{\R^d}G(t\,,x)\,\d x =\hat{G}(t\,,0)=1$, this proves (3).
	We can deduce (4) from this latter $L^1$-inclusion
	since $\hat{G}(t+s)=\hat{G}(t)\hat{G}(s)$ for all $s,t>0$.

	In order to prove the more interesting part (5), we first note 
	using \eqref{G} the following scaling property:
	$c^{d/\gamma}G(ct\,,c^{1/\gamma}x) = G(t\,,x)$
	for all $c,t>0,\ x\in\R^d.$
	Consequently, $G(t)>0$ [resp.\
	$G(t)\ge0$] for one $t>0$ if and only if $G(t)>0$  [resp.\ $G(t)\ge0$]  for all $t>0$.
	Because $\hat{G}(t\,,0)=1$ for all $t>0$, it follows from part (3)
	that $G(t)>0$ (for one hence all $t>0$)
	if and only if $\{G(t)\}_{t>0}$ is a convolution semigroup of probability densities on $\R^d$. 
	This is equivalent to saying that $G(t)>0$ for some, hence all, $t>0$
	if and only if there exists a L\'evy process
	$\{X_t\}_{t\ge0}$ on $\R^d$ such that 
	$\E\exp\{i\xi \cdot X_t\}=\exp(-t \| \xi \|^\gamma)$
	for all $t>0$ and $\xi\in\R$. Part (5) 
	is now an immediate consequences of the L\'evy-Khintchine formula \cite{Sato,Bertoin,Kyprianou}.
\end{proof}

Now let us choose and fix parameters $0\le\nu\le 1$ and $0<\beta\le d$,
and consider the constant-coefficient SPDE,
\begin{equation}\label{SHE:beta}\left[\begin{split}
	&\partial_t U(t\,,x) = \mathcal{L}_\gamma U(t\,,x)+ \dot{W}(t\,,x)&
		\forall (t\,,x)\in(0\,,\infty)\times\R^d,\\[-2mm]
	&\text{subject to }U(0\,,x) = 0&\forall x\in\R^d,
\end{split}\right.\end{equation}
where $\dot{W}$ is a centered generalized Gaussian noise and has covariance
\[
	\Cov[ \dot{W}(t\,,x)\,,\dot{W}(s\,,y)] = \kappa_{1,\nu}(t-s)
	\kappa_{d,\beta}(x-y)
	\quad\forall s,t\ge0,\ (x\,,y)\in\R^{2d},
\]
where, for every $n\in\mathbb{N}$,
$\kappa_{n,\beta}(z) = \|z\|^{-\beta}$ for all $z\in\R^n\setminus\{0\}$
and $0<\beta<n$, and $\kappa_{n,n}=\delta_0$ on $\R^n$.
Since $\kappa_{d,\beta}\in L^1_{\textit{loc}}(\R^d)$ when $0<\beta<d$,
we can view $\kappa_{d,\beta}$ as a distribution, and it is known that its Fourier transform
satisfies $\hat{\kappa}_{d,\beta} = c\kappa_{d,d-\beta}$ for an explicitly computable
number  $c=c(d\,,\beta)>0$ \cite[(12.10)]{Mattila}. In particular,
$\kappa_{d,\beta}$ is positive definite in the sense of distributions, for the entire range of 
$\beta\in(0\,,d]$,
and therefore the generalized Gaussian noise $\dot{W}$ exists according
to general theory \cite{Dudley,Minlos,DFL,Gross}.

We will frequently appeal to the fact
that, when $\beta\in(0\,,d)$,
\begin{equation}\label{Planch}
	\int_{\R^d\times\R^d} \frac{\psi(x)\phi(y)}{\|x-y\|^\beta}\,\d x\,\d y
	= \frac{c(d\,,\beta)}{(2\pi)^d}\int_{\R^d} \frac{\hat\psi(\xi)
	\overline{\hat\phi(\xi)}}{
	\|\xi\|^{d-\beta}}\,\d\xi\qquad\forall \psi,\phi\in L^1(\R^d).
\end{equation}
This can be viewed as a non-trivial 
extension of the Parseval identity from $\phi,\psi\in\mathcal{S}(\R^d)$ 
to the $L^1$-case,
and follows from Lemma 12.12 of Mattila \cite{Mattila} when $\psi$ and $\phi$ are in addition nonnegative.
Because both sides of \eqref{Planch} define bilinear forms [in $(\phi\,,\psi)$]
the above general case follows from that by considering various combinations
of $\psi^\pm$ and $\phi^\pm$.

According to Dalang 
\cite{Dalang} and Walsh \cite{Walsh},
\eqref{SHE:beta} has a random-field solution,
\begin{equation}\label{mild:beta}
	U(t\,,x) = \int_{(0,t)\times\R^d} G(t-r\,,z-x)\,W(\d r\,\d z)
	\qquad\forall t>0,\ z\in\R^d,
\end{equation}
if and only if the above Wiener integral process is defined pointwise.
This means that \eqref{mild:beta} exists as a classically defined
random field if and only if $\mathscr{R}(t)<\infty$ for every $t>0$, where
\[
	\mathscr{R}(t) = \begin{cases}
		\displaystyle\int_{(0,t)\times(0,t)}\d s\,\d r\int_{\R^d\times\R^d} 
			\frac{\d z\,\d w}{\|z-w\|^\beta} G(s\,,z)\kappa_{1,\nu}(s-r)G(r\,,w)&\text{if $0<\beta<d$},\\
		\displaystyle
		\int_{(0,t)\times(0,t)}\d s\,\d r
			\int_{-\infty}^\infty\d z\ G(s\,,z)\kappa_{1,\nu}(s-r)G(r\,,z)&\text{if $\beta=d=1$}.
	\end{cases}
\]
The integrals are $\ge0$ thanks to properties of positive-definite functions,
and the following shows when exactly $\mathscr{R}$ is finite,
equivalently when exactly \eqref{SHE:beta} has a random-field
solution.

\begin{lemma}\label{lem:R(t)}
	$\mathscr{R}(t)<\infty$ for one, hence all $t>0$, if and only if
	\begin{equation}\label{cond:gamma}
		\gamma >  \frac{\beta}{2-\nu}.
	\end{equation}
\end{lemma}

\begin{proof}
	Throughout the proof, all implied constants depend only on
	$(d\,,\beta\,,\gamma\,,\nu)$. Now apply scaling in order to see
	that it suffices to understand when exactly $\mathscr{R}(1)$
	coverges.  Consider first the case that $0<\beta<d$ and $0\le\nu<1$. In that case,
	\eqref{Planch} yields
	\begin{align*}
		\mathscr{R}(1) &\propto\int_{(0,1)\times(0,1)}\d r\,\d s
			\int_{\R^d}\frac{\d\xi}{\| \xi \|^{d-\beta}}\ 
			\frac{\exp\left\{-(s+r)\|\xi\|^\gamma\right\}}{(s-r)^\nu}\\
		&\propto \int_0^1\d s\int_0^s \d r\ |s-r|^{-\nu} (s+r)^{-\beta/\gamma}\\
		&= \int_0^1s^{1-\nu - (\beta/\gamma)} \,\d s
			\int_0^1 \d r\ |1-r|^{-\nu} (1+r)^{-\beta/\gamma},
	\end{align*}
	which is finite if and only if  $\gamma > \beta/(2-\nu).$ That case that 
	$0\le\nu<1$ and $\beta=d=1$ follows esentially the same argument.
	Next, consider the case that $\nu=1$ and $0<\beta<d$. In that case,
	\[
		\mathscr{R}(1) \propto\int_0^1\d r
		\int_{\R^d}\frac{\d\xi}{\| \xi \|^{d-\beta}}\ 
		\exp\left\{-2r\|\xi\|^\gamma\right\} \propto \int_0^1 r^{-\beta/\gamma}\,\d r
	\]
	is finite if and only if $\gamma>\beta = \beta/(2-\nu)$. And the case that $\nu=1$ and $\beta=d=1$
	is studied in almost exactly the same manner as above. This concludes the proof.
\end{proof}

In conclusion, we see that the SPDE \eqref{SHE:beta} has a random-field solution if and only if
\eqref{cond:gamma} holds, a condition that we assume throughout the remainder
of this subsection. We are
interested in the centered Gaussian process $\{X(t)\}_{t\ge0}$ where
\begin{equation}\label{X(t)}
	X(t) = U(t\,,0)\qquad\forall t>0.
\end{equation}

\begin{lemma}\label{lem:X:scale}
	Thanks to \eqref{cond:gamma}, the Gaussian process $X$ in \eqref{X(t)}
	is $\alpha$-self similar for
	\begin{align}\label{alpha}
		\alpha =  1 - \frac12\left( \nu + \frac{\beta}{\gamma}\right).
	\end{align}
\end{lemma}

\begin{proof}
	Suppose $0\le\nu<1$ and either $0<\beta<d$ or
	$\beta=d=1$. Then, in either case,
	the $L^2$-isometry properties of
	Wiener integrals and \eqref{Planch} together
	yield the following for every $t>0$ and $h\ge0$:
	\begin{equation}\label{Cov:SHE:1}\begin{split}
		\Cov[X(t+h)\,,X(t)] 
			&= L_0\int_0^{t+h}\d s\int_0^t\d r\int_{\R^d}\d\xi\
			\frac{\exp\{-(s+h+r)\|\xi\|^\gamma\}}{|s+h-r|^\nu\|\xi\|^{d-\beta}}\\
		&=L_1\int_0^{t+h}\d s\int_0^t\d r\ \frac{1}{|s+h-r|^\nu(s+h+r)^{\beta/\gamma}},
	\end{split}\end{equation}
	where $L_j=L_j(d\,,\beta\,,\gamma)$ [$j=0,1$] are strictly positive but otherwise
	independent of all else. This readily yields
	\begin{equation}\label{Cov:scale}
		\Cov[X(ct+ch)\,,X(ct)] = c^{2\alpha}\Cov[X(t+h)\,,X(t)]\qquad
		\forall c,t>0,\ h\ge0,
	\end{equation}
	where $\alpha$ is defined in Lemma \ref{lem:X:scale}. It remains to study the case
	that $\nu=1$, in which case the noise
	$\dot{W}$ in \eqref{SHE:beta} is white in its time variable.	
	In that case, one computes similarly as was done for \eqref{Cov:SHE:1} and obtains
	the following for every $t>0$ and $h\ge0$:
	\begin{equation*}\begin{split}
		\Cov[X(t+h)\,,X(t)] &= L_0 \int_0^t\d s\int_{\R^d}\d\xi\
			\frac{\exp\{-(2s+h)\|\xi\|^\gamma\}}{\|\xi\|^{d-\beta}}
		= L_2\int_0^t \frac{\d s}{(2s+h)^{\beta/\gamma}},
	\end{split}\end{equation*}
	where $L_2=L_2(d\,,\beta\,,\gamma)$ [$j=0,1$] is strictly positive but otherwise
	independent of all else. 
	This readily yields \eqref{Cov:scale} as well, and completes the proof.
\end{proof}

Next, we verify the continuity and an associated moment estimate for $X$.
Recall that we are assuming \eqref{cond:gamma} in order to ensure
that $X$ is well defined.

\begin{lemma}\label{lem:X:cont}
	The Gaussian process $X$ in \eqref{X(t)}
	is continuous. Moreover, there exists $K=K(\gamma\,,\nu\,,\beta\,,d)>0$
	such that $\E\left(|X(t+h)-X(t)|^2\right) \le Kh^{2\alpha}$ uniformly for all $t>0$ and $h\ge0$, where $\alpha$ is the same constant as in \eqref{alpha}.
\end{lemma}

\begin{proof}
First, consider the case that $0<\nu \le 1$.
In that case, we may apply \eqref{Planch} twice to write
\begin{align*}
	&\E\big(|X(t+h)-X(t)|^2\big)\\
	&\propto \int_{\R} \d \tau \int_{\R^d} \d \xi \left| (\e^{-i(t+h)\tau}-\e^{-(t+h)\|\xi\|^\gamma}) - (\e^{-it\tau}-\e^{-t\|\xi\|^\gamma}) \right|^2 \frac{\|\xi\|^{\beta-d}|\tau|^{\nu-1}}{\|\xi\|^{2\gamma}+|\tau|^2},
\end{align*}
where the constant of proportionality depends only on $(d\,,\beta)$. 
We have also used the fact that the Fourier transform of 
$\R\ni r\mapsto f(r) = \exp\{-(t-r) \|\xi\|^\gamma\}\1_{[0,t]}(r)$ is 
$\hat{f}(\tau)=(\e^{-it\tau}-\e^{-t\| \xi \|^{\gamma}})/(\|\xi \|^\gamma - i\tau)$.
In particular, we may apply the mean value theorem to see that
\begin{align*}
	\left| (\e^{-i(t+h)\tau}-\e^{-(t+h)\|\xi\|^\gamma}) - (\e^{-it\tau}-\e^{-t\|\xi\|^\gamma}) \right| \le h(\|\xi\|^{\gamma}+|\tau|)
\end{align*}
uniformly for all $t>0$ and $h \ge 0$.
It is clear that the left-hand side above is always bounded by $4$.
It follows that
\begin{align*}
	\E\big(|X(t+h)-X(t)|^2\big)
	\lesssim \int_{\R} \d \tau \int_{\R^d} \d \xi \left[16 \wedge \left(h^2(\|\xi\|^{2\gamma}+|\tau|^2)\right)\right] \frac{\|\xi\|^{\beta-d}|\tau|^{\nu-1}}{\|\xi\|^{2\gamma}+|\tau|^2}.
\end{align*}
By scaling, the last integral is equal to $Kh^{2-\nu-(\beta/\gamma)}$, where
\[
K = \int_{\R} \d \tau \int_{\R^d} \d \xi \left[16 \wedge (\|\xi\|^{2\gamma}+|\tau|^2)\right] \frac{\|\xi\|^{\beta-d}|\tau|^{\nu-1}}{\|\xi\|^{2\gamma}+|\tau|^2}
\]
is a positive finite real number since $\beta+(\nu-2)\gamma<0$ under \eqref{cond:gamma}.

Next, consider the case that $\nu = 0$.
In that case, we may use \eqref{Planch} to write
\begin{align*}
	&\E\big(|X(t+h)-X(t)|^2\big)\\
	&\propto \int_{\R^d}  \d \xi \, \|\xi\|^{\beta-d} \left[ 
		\int_0^{t+h} \d s \, \e^{-(t+h-s)\|\xi\|^\gamma} - 
		\int_0^t \d s\, \e^{-(t-s)\|\xi\|^\gamma} \right]^2
		= Q_1+2Q_2+Q_3,
\end{align*}
where
\begin{align*}
	Q_1 &= \int_{\R^d} \d \xi \, \|\xi\|^{\beta-d} \left[ \int_0^t \d s  
		\left( \e^{-(t+h-s)\|\xi\|^\gamma} - \e^{-(t-s)\|\xi\|^\gamma} 
		\right) \right]^2,\\
	Q_2 &= \int_{\R^d} \d \xi \, \|\xi\|^{\beta-d} \int_0^t \d s 
		\left( \e^{-(t+h-s)\|\xi\|^\gamma} - \e^{-(t-s)\|\xi\|^\gamma} \right) 
		\int_t^{t+h} \d r \, \e^{-(t+h-r)\|\xi\|^\gamma},\\
	Q_3 &= \int_{\R^d} \d \xi \, \|\xi\|^{\beta-d} \left[ \int_t^{t+h} \d s \, 
		\e^{-(t+h-s)\|\xi\|^\gamma} \right]^2.
\end{align*}
In particular, we may evaluate the integral in $s$ and then use scaling to deduce that
\begin{align*}
	Q_1 &= \int_{\R^d}\d \xi\,\|\xi\|^{\beta-d-2\gamma} 
		\left( 1 - \e^{-h\|\xi\|^\gamma} \right)^2 \left( 1 - \e^{-t\|\xi\|^\gamma} \right)^2\\
	& \le \int_{\R^d}\d \xi\, \|\xi\|^{\beta-d-2\gamma} 
		\left( 1 - \e^{-h\|\xi\|^\gamma} \right)^2 \propto h^{2-(\beta/\gamma)} = h^{2\alpha}.
\end{align*}
Similarly,
\begin{align*}
	|Q_2| &= \int_{\R^d}\d \xi\, \|\xi\|^{\beta-d-2\gamma} 
		\left( 1 - \e^{-h\|\xi\|^\gamma} \right)^2 \left( 1 - \e^{-t\|\xi\|^\gamma} \right)
		\lesssim h^{2\alpha},\\
	Q_3 &= \int_{\R^d}\d \xi\,\|\xi\|^{\beta-d-2\gamma} 
		\left( 1 - \e^{-h\|\xi\|^\gamma} \right)^2 \propto h^{2\alpha}.
\end{align*}
The implied constants are all independent of $t$ and $h$.
Combine the estimates for $Q_1$, $Q_2$ and $Q_3$ to finish the proof.
\end{proof}

We complete this section by showing that the entirety of the statements of
Theorem \ref{th:main} and Corollary \ref{cor:main} apply
to the Gaussian process $t\mapsto X(t)=U(t\,,0)$,
where $U$ solves the SPDE \eqref{SHE:beta}. 
The following does that when $0<\nu\le 1$.

\begin{lemma}\label{lem:X:SLND}
	Assume \eqref{cond:gamma} with $0 < \nu \le 1$ and either $0<\beta<d$ or $\beta=d=1$.
	Then the Gaussian process $X$ in \eqref{X(t)}
	satisfies \eqref{SLND} with $\alpha$ given by \eqref{alpha}.
\end{lemma}

\begin{proof}
	We follow carefully the proof of Lee and Xiao \cite[Lemma 7.3]{LX23}, applicable
	when $\gamma\in(0\,,2]$, $0 < \nu \le 1$ and $0<\beta < d$, in order extend it to the present setting.
	
	It suffices to show the existence of a number $\mathscr{l}_{_X}>0$ such that
	\begin{align*}\textstyle
		\E\left(\left| X(t) - \sum_{j=1}^n a_j X(s_j)\right|^2 \right)
		\ge \mathscr{l}_{_X} (t-s_n)^{(\gamma-\beta)/\gamma}
	\end{align*}
	uniformly for all integers $n\ge 1$, $0\le s_1 < \cdots < s_n < t$ and $a_1,\dots, a_n \in \R$.
	To this end, we first use \eqref{Planch} as in the proof of Lemma \ref{lem:X:cont} to see that
	\begin{align}\label{SLND:SHE}
		&\textstyle\E\left(\left| X(t) - \sum_{j=1}^n a_j X(s_j)\right|^2 \right)\\\notag
		&\propto\int_{\R} \d\tau\int_{\R^d}\d\xi\,
			\textstyle
			\left| (\e^{-it\tau} - \e^{-t \|\xi\|^\gamma}) - \sum_{j=1}^n a_j 
			(\e^{-is_j\tau}-\e^{-s_j \|\xi \|^\gamma}) \right|^2 
			\displaystyle\frac{\|\xi\|^{\beta-d}|\tau|^{\nu-1}}{\|\xi\|^{2\gamma} + |\tau|^2},
	\end{align}
	where the implied constant depends only on $(d\,,\beta)$. 
%Also,
%	in the last equality, we have used the simple fact that the Fourier transform of the function 
%	$\R\ni r\mapsto f(r) = \exp\{-(t-r) \|\xi\|^\gamma\}\1_{[0,t]}(r)$ is 
%	$\hat{f}(\tau)=(\e^{-it\tau}-\e^{-t\| \xi \|^{\gamma}})/(\|\xi \|^\gamma - i\tau)$.
	
	Next, we choose and fix two smooth functions $\phi: \R \to \R$ 
	and $\psi:\R^d\to\R$ that are respectively supported in $[-1\,,1]$ 
	and $[-1\,,1]^d$ and satisfy $\phi(0) = \psi(0) = 1$.
	For every $\varepsilon>0$, $x\in\R$, and $y\in\R^d$, define 
	\begin{align*}
		\phi_\varepsilon(x) = \varepsilon^{-1}\phi(\varepsilon^{-1}x),
		\qquad
		\psi_\varepsilon(y) = \varepsilon^{-d/\gamma}\psi(\varepsilon^{-1/\gamma}y).
	\end{align*}
	Define
	\begin{equation}\label{epsilon}
		\varepsilon =t-s_n,
	\end{equation}
	and
	\begin{align*}
		&\mathcal{I} := \int_{\R} \d\tau \int_{\R^d}\d \xi\, 
			\textstyle\left[ (\e^{-it\tau} - \e^{-t\| \xi \|^\gamma}) 
			- \sum_{j=1}^n a_j (\e^{-is_j\tau}-\e^{-s_j \|\xi \|^\alpha}) \right] 
			\e^{it\tau} \hat{\phi}_\varepsilon(\tau) \hat{\psi}_\varepsilon(\xi)\\
		&= \int_{\R} \d\tau \int_{\R^d}\d \xi\, 
			\textstyle\left[ (\e^{-it\tau} - \e^{-t\| \xi \|^\gamma}) 
			- \sum_{j=1}^n a_j (\e^{-is_j\tau}-\e^{-s_j \|\xi \|^\alpha}) \right] 
			\e^{it\tau} \hat{\phi}(\varepsilon \tau) \hat{\psi}(\varepsilon^{1/\gamma} \xi),
	\end{align*}
	where $\hat{\phi}_\varepsilon$ and $\hat{\psi}_\varepsilon$
	denote the Fourier transforms of $\phi_\varepsilon$ and $\psi_\varepsilon$,
	as is customary. We apply inverse Fourier transforms to compute the
	$\d\tau$-integral
	and simplify the expression for $\mathcal{I}$ as follows:
	\[
		\mathcal{I} \propto  \int_{\R^d}
		\textstyle\left[ 
		(\phi_{\varepsilon}(0)-\phi_{\varepsilon}(t) \e^{-t \| \xi\|^\gamma}) 
		- \sum_{j=1}^n a_j (\phi_\varepsilon(t-s_j) - \phi_\varepsilon(t) 
		\e^{-s_j\| \xi \|^\gamma}) \right] \hat{\psi}(\varepsilon^{1/\gamma} \xi)\,\d\xi,
	\]
	where the constant of proportionality is $(2\pi)^d$.
	Because $\phi$ is supported in $[-1\,,1]$, it follows that
	$\phi_\varepsilon(t) = \phi_\varepsilon(t-s_j)=0$ 
	for every $1 \le j \le n$ since $t/\varepsilon \ge 1$, $(t-s_j)/\varepsilon \ge 
	(t-s_n)/\varepsilon = 1$. Therefore, the inversion formula yields
	$\mathcal{I} \propto \phi_\varepsilon(0) \int_{\R^d} \hat{\psi}_\varepsilon (\xi) \,\d \xi
	\propto \phi_\varepsilon(0) \psi_\varepsilon (0) = \varepsilon^{-1-(d/\gamma)},$
	where the constants of proportionality depend only on $d$. 
	Therefore, the definition of $\mathcal{I}$, \eqref{SLND:SHE},
	and the Cauchy-Schwarz inequality together imply that
	\begin{align}
		&\varepsilon^{-2-(2d/\gamma)}\propto \mathcal{I}^2 \label{I^2}\\\notag
		&{\textstyle
			\lesssim \E\left( \left| X(t) - \sum_{j=1}^n a_j X(s_j)\right|^2 \right)}
			\int_{\R}\d\tau
			\int_{\R^d}\d\xi\
			|\hat{\phi} (\varepsilon\tau)|^2 |\hat{\psi}(\varepsilon^{1/\gamma}\xi)|^2 
			\left(\frac{\|\xi\|^{2\gamma} + |\tau|^2}{|\tau|^{\nu-1}\|\xi\|^{\beta-d}}\right).
	\end{align}
	Because $\phi$ and $\psi$ are test functions of rapid decrease, we
	may apply scaling to see that the preceding integral is equal to some positive real number
	$c=c(d\,,\gamma\,,\beta\,,\phi\,,\psi)$ times $\varepsilon^{-4-(2d/\gamma)+(\beta/\gamma)+\nu}$.
	Therefore, \eqref{I^2} yields
	\[\textstyle
		\E\left(\left| X(t) - \sum_{j=1}^n a_j X(s_j)\right|^2 \right)
		\gtrsim\varepsilon^{2-\nu-(\beta/\gamma)} = (t-s_n)^{2\alpha},
	\]
	thanks to \eqref{epsilon}, where the implied constant is a number that depends only on 
	$(d\,,\gamma\,,\beta\,,\phi\,,\psi)$ and in particular not on 
	$t,s_1,\ldots,s_n$. This completes the proof.
\end{proof}

\begin{OP}
	In the case $\nu=0$, one can show in like manner that
	\begin{align*}
		&\textstyle\E\left(\left| X(t) - \sum_{j=1}^n a_j X(s_j)\right|^2 \right)\\
%		&\textstyle\propto \int_{\R^d} \d \xi \,\|\xi\|^{\beta-d} 
%			\left[ \int_0^\infty \d r \left( \e^{-(t-r)\|\xi\|^\gamma} 
%			\1_{[0,t]}(r) - \sum_{j=1}^n a_j \e^{-(s_j-r)\|\xi\|^\gamma} 
%			\1_{[0,s_j]}(r) \right) \right]^2\\
		&\textstyle=\int_{\R^d} \d \xi \,\|\xi\|^{\beta-d-2\gamma} 
			\left[ (1-\e^{-t\|\xi\|^\gamma}) - \sum_{j=1}^n a_j(1-\e^{-s_j\|\xi\|^\gamma}) \right]^2.
	\end{align*}
	Let $\phi:\R^d\to\R$ be a test function and define
	\[\textstyle
		\mathcal{I} = \int_{\R^d} \d \xi \,\left[ (1-\e^{-t\|\xi\|^\gamma}) 
		- \sum_{j=1}^n a_j(1-\e^{-s_j\|\xi\|^\gamma}) \right] \hat{\phi}(\xi).
	\]
	Then
	\[\textstyle
		\mathcal{I} = \left[ (2\pi)^d\phi(0) - \int_{\R^d} \e^{-t\|\xi\|^\gamma} 
		\hat{\phi}(\xi) \d \xi \right] - \sum_{j=1}^n a_j \left[
		(2\pi)^d\phi(0)-\int_{\R^d} \e^{-s_j\|\xi\|^\gamma}\hat{\phi}(\xi)\, \d \xi\right].
	\]
	It seems that there is no natural choice of $\phi$ that makes all the terms in 
	$\sum_{j=1}^n (\cdots)$ vanish.
	Therefore, it is unclear to us whether or not $X$ satisfies \eqref{SLND} when $\nu=0$,
	and the validity (or not) of \eqref{SLND} when $\nu=0$ is  an open problem.
\end{OP}

\subsection{Comparison of critical exponents}\label{sec:comparison}

Recall the process $X$ defined by \eqref{X(t)} which is $\alpha$-self similar 
and satisfies \eqref{SLND} with the index $\alpha$ given by \eqref{alpha} 
(see Lemmas \ref{lem:X:scale} and \ref{lem:X:SLND}).
Thanks to Theorem \ref{th:main}, the critical boundary crossing exponent function 
$\lambda_X(c) = \lambda(c\,,X)$ is well defined for $X$.
The goal of this section is to establish the following inequality which 
relates $\lambda_X$ to the exponent 
$\lambda_B(c) = \lambda(c\,,B)$ of a fractional Brownian 
motion $B$ with the same index $\alpha$.

\begin{theorem}\label{th:crit:exp}
	If $\nu=1$ and either $0<\beta<d$ or $\beta=d=1$, then
	\[
		\lambda_X(c) \le \lambda_B(c/K_0)
		\qquad\forall c>0,
	\]
	where $B$ denotes an $\fBm$ with index $\alpha = \frac12 (1-\beta\gamma^{-1})$,
	and 
	\begin{align}\label{K_0}
		K_0^2 = \frac{c(d\,,\beta)}{2(2\pi)^d} 
		\int_{\R^d} \left[ \left( 1-\e^{-\|\xi\|^\gamma} \right)^2  
		+\left( 1-\e^{-2\|\xi\|^\gamma} \right)  \right] 
		\frac{\d \xi}{\|\xi\|^{d+\gamma-\beta}}.
	\end{align}
\end{theorem}

\begin{OP}
	Is $\lambda_X(c) = \lambda_B(c/K_0)$?
\end{OP}
We are not able to produce any bounds in the reverse direction, and have even
failed to numerically evaluate these boundary crossing exponents reliably in order
to know whether we can hope to upgrade the proceeding to a conjecture.

In order to prove Theorem \ref{th:crit:exp}, 
let us introduce a white noise $W_1$ on $\R^d$ independent of the noise $W$ in \eqref{SHE:beta} and define the centered Gaussian process $S=\{S(t)\}_{t \ge 0}$ by
\begin{align}\label{S}
	S(t) = \sqrt{\frac{c(d\,,\beta)}{2(2\pi)^{d}}} \int_{\R^d} \frac{1-\e^{-t\|\xi\|^\gamma}}{\|\xi\|^{(d+\gamma-\beta)/2}} W_1(\d\xi) \qquad \forall t \ge 0,
\end{align}
where $c(d\,,\beta)$ is the number in \eqref{Planch}.
The processes $X$ and $S$ are independent.

\begin{lemma}\label{lem:S:cont}
	The process $S$ is a.s.~continuous on $\R_+$ and $\cC^\infty$ on $(0\,,\infty)$.
\end{lemma}

\begin{proof}
	For every $t,h>0$,
	\begin{align*}
		\E\big( |S(t+h)-S(t)|^2\big) 
			&\propto \int_{\R^d} \e^{-2t\|\xi\|^\gamma}
			\left( 1 - \e^{-h\|\xi\|^\gamma} \right)^2 \|\xi\|^{\beta-\gamma-d}\, \d \xi\\
		&\lesssim h^2 \int_{\R^d}  \e^{-2t\|\xi\|^\gamma} \|\xi\|^{\beta+\gamma-d}\, \d \xi
			\propto \frac{h^2}{t^{1+(\beta/\gamma)}},
	\end{align*}
	where the implied constants depend only on $(d\,,\beta\,,\gamma)$.
	In particular, 
	\begin{equation}\label{S:tmpsup}
		\sup_{t\ge h}\E\big( |S(t+h)-S(t)|^2\big) \lesssim h^{1-(\beta/\gamma)}\quad
		\forall h>0,
	\end{equation}
	where the implied constant depends only on $(d\,,\beta\,,\gamma)$.
	Also, by \eqref{S} and scaling, $\E( |S(t)|^2)\propto t^{1-(\beta/\gamma)}$
	uniformly for every $t>0$. Therefore,
	\begin{align*}
		\E\big( |S(t+h)-S(t)|^2\big) &\le 2\E\big(|S(t+h)|^2\big)+2\E\big(|S(t)|^2\big)\\
		&\lesssim (t+h)^{1-(\beta/\gamma)}+ t^{1-(\beta/\gamma)}
			\qquad\forall t,h\ge0.
	\end{align*}
	This and \eqref{S:tmpsup} together yield
	\[
		\E\big( |S(t+h)-S(t)|^2\big) \lesssim h^{1-(\beta/\gamma)}
		\quad\forall t,h>0,
	\]
	where the implied constant depends only on $(d\,,\beta\,,\gamma)$. Thanks to the Kolmogorov
	continuity theorem and the fact that moments of Gaussian random variables
	are determined by their variance, it follows that that $S\in\cC^a_{\textit{loc}}(\R_+)$
	a.s.\ for every $a\in(0\,, \frac12(1-\beta\gamma^{-1}))$. In particular, $S$
	is continuous on $\R_+$. Next define
	\[
		S^{(n)}(t)
		= \sqrt{\frac{c(d\,,\beta)}{2(\pi)^d}}(-1)^{n+1} t^n \int_{\R^d} 
		\|\xi\|^{\frac12(\beta-d)+(n-\frac12)\gamma}\e^{-t\|\xi\|^\gamma}
		\,W_1(\d\xi) \quad\forall t > 0,\ n\in\N,
	\]
	where $W_1$ is the same 
	white noise that was used in \eqref{S} to define $S$. The process $S^{(n)}$ is well defined,
	for example because
	\begin{align*}
		\E\left(|S^{(n)}(t)|^2\right) \propto \int_{\R^d} 
		\|\xi\|^{\beta-d+(2n-1)\gamma} \e^{-2t\|\xi\|^\gamma}\,\d\xi
		<\infty\quad\forall t > 0,\ n\in\N.
	\end{align*}
	Define
	\[
		k(t\,,\xi) = 
		\sqrt{\frac{c(d\,,\beta)}{2(2\pi)^d}} \left(\frac{ 1 - \e^{-t\|\xi\|^\gamma}}{%
		\|\xi\|^{(d+\gamma-\beta)/2}}\right)\qquad\forall t>0,\
		\xi\in\R,
	\]
	and observe that 
	\[
		S(t)=\int_{\R^d} k(t\,,\xi)\,W_1(\d\xi),\quad
		S^{(n)}(t) = \int_{\R^d} \partial^n_t k(t\,,\xi)\, W_1(\d\xi).
	\]
	Thanks to the preceding $L^2$-computations, we may apply a stochastic Fubini
	argument twice in order to see that, for all non random $f\in\cC^\infty_c(0\,,\infty)$,
	\begin{align*}
		&\int_0^\infty S^{(n)}(t)f(t)\,\d t = \int_{\R^d} W_1(\d\xi)
			\int_0^\infty\d t\ f(t)\partial^n_t k(t\,,\xi)\\
		&= (-1)^n\int_{\R^d} W_1(\d\xi)
			\int_0^\infty\d t\ \frac{\d^n f(t)}{\d t^n} k(t\,,\xi)
			= (-1)^n\int_0^\infty S(t)\frac{\d^n f(t)}{\d t^n}\,\d t
			\quad\text{a.s.}
	\end{align*}
	It follows that $S^{(n)}$ is the $n$th weak derivative of $S$, and hence it remains to 
	prove that $S^{(n)}$ is continuous [up to a modification]. With that end in mind note that
	\[
		\E\left(|S^{(n)}(t) - S^{(n)}(s)|^2\right)\lesssim Q_1+Q_2,
	\]
	where the implied constant is universal, and
	\begin{align*}
		Q_1&= (t^n-s^n)^2 \int_{\R^d} \|\xi\|^{\beta-d+(2n-1)\gamma} \e^{-2t\|\xi\|^\gamma}\d\xi,\\
		Q_2&= s^{2n}\int_{\R^d} \|\xi\|^{\beta-d+(2n-1)\gamma}
			\left( \e^{-s\|\xi\|^\gamma} - \e^{-t\|\xi\|^\gamma}\right)^2\,\d\xi.
	\end{align*}
	Choose and fix two arbitrary numbers $\varepsilon$ and $M$ that satisfy
	$0<\varepsilon<1<M$. Since
	$t^n-s^n=n^{-1}\int_s^t x^{n-1}\,\d x$,
	\[
		Q_1 \le n^{-2}M^{2n-2}(t-s)^2\int_{\R^d} 
		\|\xi\|^{\beta-d+(2n-1)\gamma} \e^{-2\varepsilon\|\xi\|^\gamma}\d\xi
		\lesssim (t-s)^2,
	\]
	uniformly for all $s,t\in[\varepsilon\,,M]$. Also,
	\begin{align*}
		Q_2&\le M^{2n}\int_{\R^d} \|\xi\|^{\beta-d+(2n-1)\gamma}
			\left( 1-\e^{-|t-s| \|\xi\|^\gamma}\right)^2 
			\e^{-2\varepsilon\|\xi\|^\gamma}\d\xi\\
		&\le M^{2n}(t-s)^2\int_{\R^d} \|\xi\|^{\beta-d+(2n+1)\gamma}
			 \e^{-2\varepsilon\|\xi\|^\gamma}\d\xi \propto (t-s)^2,
	\end{align*}	
	uniformly for all $s,t\in[\varepsilon\,,M]$.
	The implied constants above depend only on 
	$(d\,,\beta\,,\gamma\,,n\,,\varepsilon\,,M)$.
	Thus, we can deduce from Kolmogorov's continuity theorem 
	that $S^{(n)}$ is continuous on $(0\,,\infty)$. 
	 This concludes the proof of Lemma \ref{lem:S:cont}.
\end{proof}

\begin{lemma}\label{lem:X:fBm}
	If $\nu = 1$ and either $0<\beta<d$ or $\beta=d=1$, then
	$B(t):= \frac{X(t)+S(t)}{K_0}$ [$t\ge 0$]
	is an $\fBm$  with index $\frac12(1-\beta\gamma^{-1})$, 
	where $K_0$ was defined in \eqref{K_0}.
\end{lemma}

\begin{proof}
	For every $t, h \ge 0$ we may write $X(t+h) - X(t) = I_1+I_2$ where
	\begin{align*}
		I_1&= \int_{(0,t)\times\R^d} \left[G(t+h-r\,,z) - G(t-r\,,z)\right] W(\d r\,\d z),\\
		I_2&=\int_{[t, t+h) \times\R^d} G(t+h-r\,,z)\, W(\d r\,\d z).
	\end{align*}
	Clearly, $I_1$ and $I_2$ are independent since the noise $W$ is white in time when $\nu=1$.
	Thanks to \eqref{Planch} and \eqref{S}, we deduce that
	\begin{align*}
		\Var(I_1)&=\frac{c(d\,,\beta)}{(2\pi)^d}\int_0^t \d r 
			\int_{\R^d} \d \xi \left(\e^{-(t+h-r)\|\xi\|^\gamma} - 
			\e^{-(t-r)\|\xi\|^\gamma}\right)^2\|\xi\|^{\beta-d}\\
%		&=\frac{c(d\,,\beta)}{2(2\pi)^d} \int_{\R^d}  \left( 1-\e^{-h\|\xi\|^\gamma} \right)^2
%			\left( 1 - \e^{-2t\|\xi\|^\gamma} \right)
%			\frac{\d \xi}{\|\xi\|^{d+\gamma-\beta}}\\
		&=\frac{c(d\,,\beta)}{2(2\pi)^d} \int_{\R^d} 
			\left( 1-\e^{-h\|\xi\|^\gamma} \right)^2\frac{\d \xi}{\|\xi\|^{d+\gamma-\beta}}
			- \E\left( |S(t+h)-S(t)|^2 \right),
	\end{align*}
	and
	\begin{align*}
		\Var(I_2) &= \frac{c(d\,,\beta)}{(2\pi)^d} \int_t^{t+h} \d r \int_{\R^d} 
			\d\xi \, \e^{-2(t+h-r)\|\xi\|^\gamma} \|\xi\|^{\beta-d}\\
			&= \frac{c(d\,,\beta)}{2(2\pi)^d} \int_{\R^d} 
			\left( 1 - \e^{-2h\|\xi\|^\gamma}\right) \frac{\d\xi}{\|\xi\|^{d+\gamma-\beta}}.
	\end{align*}
	Then, we may combine these two expressions, use independence 
	of $I_1$ and $I_2$, and apply scaling to find that
	$\E( |X(t+h)-X(t)|^2 )
		= K_0^2 h^{1-(\beta/\gamma)} - \E( |S(t+h)-S(t)|^2 ).$
	Lemma \ref{lem:S:cont} implies that $B=\{B(t)\}_{t \ge 0}$ 
	is a continuous centered Gaussian process and the preceding 
	shows that $\E(|B(t+h)-B(t)|^2) = h^{1-(\beta/\gamma)}$.
	Therefore, $B$ is a fractional Brownian motion with index
	$\frac12 (1-\beta\gamma^{-1})$.
\end{proof}

\begin{proof}[Proof of Theorem \ref{th:crit:exp}]
	Since $X$ and $S$ are independent centered Gaussian processes, 
	we may apply conditionally Anderson's shifted-ball inequality 
	\cite{Anderson55} to see that for every $n \in \N$,
	\begin{align*}
		\P\left\{ |X(t)| \le c\, t^\alpha\ \forall t \in [1\,,n] \right\}
			&=\P\left\{ K_0^{-1}|X(t)| \le K_0^{-1} c\, t^\alpha\ \forall t \in [1\,,n] \right\}\\
		&\ge \P\left\{ K_0^{-1}|X(t)+S(t)| \le K_0^{-1} c\,t^\alpha\ \forall t \in [1\,,n] \right\}\\
		&= \P\left\{ |B(t)| \le K_0^{-1} c\,t^\alpha\ \forall t \in [1\,,n] \right\};
	\end{align*}
	see Lemma \ref{lem:X:fBm} for the last line.
	Apply a logarithm to both sides, divide by $\log n$, and let $n \to \infty$ in order to deduce
	from Proposition \ref{pr:X} that $\lambda_X(c) \le \lambda_B(c/K_0)$ for all $c>0$.
\end{proof}

\section{On simulation analysis}

One of our original goals was to devise relatively fast simulation methods for computing
the boundary crossing exponent $\lambda$ for a self-similar Gaussian process,
such as $\fBm$ or the solution to a nice linear
stochastic PDEs, that satisfies \eqref{SLND} and has good
H\"older regularity properties. At present we are only aware of
Monte-Carlo type methods. Because we are trying to estimate large-deviation
exponents, those methods tend to be computationally expensive and we could
not realistically use them on a standard personal laptop computer.
Nevertheless, we hope that good methods might exist.
In anticipation of that, we will address a theoretical matter
that we feel most, if not all, simulation analyses will need resolved. 

In order to be concrete, from now on, we assume that $X$ is an $\fBm$ with Hurst index $\alpha$, though all of what follows applies a fair bit more generally.

For every $c>0$ define
\[
	T^*_c = \inf\{ \e^n : \, n\in\N\,,\, |X(\e^n)| > c\e^{n\alpha}\}.
\]
One can readily adjust and reuse the very
same argument that was used to
prove Theorem \ref{th:main} in order to prove the following
whose details are therefore omitted.

\begin{proposition}	
	For every $c>0$ there exists a real number $\lambda^*(c)=\lambda^*(c\,,X)\ge0$ such that
	\[
		\P\{ T^*_c> \e^n \} = \e^{-n(\lambda^*(c)+\mathcal{o}(1))}\quad\text{as $n\to\infty$}.
	\]
\end{proposition}

Because $T^*_c \ge T_{c,\alpha}$ it follows that
$\lambda^*(c) \le\lambda(c)$ for every $c>0$. 
We do not know whether or not the inequality is
sharp, and therefore ask the following.

\begin{OP}\label{OP2}
	Is $\lambda^*(c)=\lambda(c)$ for all $c>0$?
\end{OP}

Fortunately, there are variations of this problem that we know how to solve. 
In order to describe the details, let us choose and fix an unbounded sequence 
of numbers $1\le t_1<t_2<\ldots$, and consider
the stopping times
\[
	\hat{T}_c = \inf\{ t_n : \, n\in\N\,,\, |X(t_n)| > ct_n^\alpha\}
	\qquad\forall c>0.
\]
Since $\hat{T}_c \ge T_{c,\alpha}$, we always have
\begin{equation}\label{LB}
	\P\{ \hat{T}_c > t_n \} \ge \P\{ T_{c,\alpha} > t_n\}.
\end{equation}
The following is a way to turn the preceding inequality in reverse direction.

\begin{lemma}\label{lem:discrete}
	Suppose that $t_{n+1}-t_n = \mathcal{o}(t_n)$ as $n\to\infty$.
	Then, there exists a constant $K >1$ with the following property:
	Whenever $m(n)$ is an integer in $(1\,,n)$ for every $n\in\N$,
	\[
		\P\{\hat{T}_c > t_n\}
		\le \P\{ T_{c+\varepsilon,\alpha} > t_n/t_{m(n)}\}
		+ K\sum_{j=m(n)}^n \exp\left(- \frac{\varepsilon^2}{K}
		\left|\frac{t_j}{t_{j+1}-t_j}\right|^{2\alpha} \right),
	\]
	for every $c,\varepsilon>0$ and $n\in\N$.
\end{lemma}

\begin{proof}
	Clearly, $\P\{\hat{T}_c>t_n\}$ is equal to
	\begin{align}\notag
		&\P\left\{ |X(t_j)|\le  ct_j^\alpha\ \forall j=1,\ldots,n\right\}\\
			&\le \P\left\{|X(t)|\le (c+\varepsilon)t^\alpha\ \forall t\in[t_{m(n)}\,,t_n]\right\}
			\label{T1}\\\notag
		&\textstyle\quad + \P\left\{\exists j\in\{m(n)\,,\ldots,n\}:\ 
			\sup_{t\in[t_j,t_{j+1}]} |X(t)-X(t_j)| \ge \varepsilon t_j^\alpha\right\}.
	\end{align}
	The self-similarity of $X$ implies that the first term on the right-hand side of \eqref{T1}
	is equal to
	\[
		\P\{|X(t)|\le (c+\varepsilon)t^\alpha\ 
		\forall t\in[1\,,t_n/t_{m(n)}]\}
		=\P\{T_{c+\varepsilon,\alpha}>t_n/t_{m(n)}\}.
	\]
	Therefore, it suffices to study the second term on the right-hand side of \eqref{T1}.
	
	Choose and fix an integer $j\ge1$. The Gaussian process $Y(t) = X(t)-X(t_j)$ 
	[$t_j \le t\le t_{j+1}$] satisfies
	$\E( |Y(t)-Y(s)|^2) =  |t-s|^{2\alpha}$
	uniformly for every integer $j\ge1$ and all $s,t\in[t_j\,,t_{j+1}]$.
	Let $N$ denote the metric entropy of $[t_j\,,t_{j+1}]$ with respect to the
	metric $d(t\,,s)=\|Y(t)-Y(s)\|_2$ for every $s,t\in[t_j\,,t_{j+1}]$. That is,
	$N(r)$ denotes the minimum number of closed $d$-balls of radius $r>0$
	needed to cover $[t_j\,,t_{j+1}]$. The preceding
	shows that $N(r) \le (t_{j+1}-t_j) r^{-1/\alpha}$ for all $r\in (0\,,(t_{j+1}-t_j)^\alpha )$.
	It also shows that the $d$-diameter $D$ of $[t_j\,,t_{j+1}]$ satisfies $D\le (t_{j+1}-t_j)^\alpha$.
	It follows from metric entropy arguments -- see Dudley
	\cite{Dudley1967} -- and the definition of the process $Y$ that 
	there exist universal numbers $C_1,C_2>0$ independently of $j$ such that
	$\E\sup_{t\in[t_j,t_{j+1}]}|X(t)-X(t_j)| $ is equal to
	\begin{align}\begin{split}\label{Dudley}
		\E\sup_{t\in[t_j,t_{j+1}]}|Y(t)| 
		&\le C_1 \int_0^{C_1(t_{j+1}-t_j)^\alpha} |\log_+((t_{j+1}-t_j) r^{-1/\alpha})|^{1/2}\,\d r\\
		&\le C_2 (t_{j+1}-t_j)^\alpha,
	\end{split}\end{align}
	where $\log_+(x)=\log(\e+x)$ for all $x\ge0$. 
	Thanks to concentration of measure
	(see for example Ledoux \cite[Chapter 2]{Ledoux}),
	for all $z>0$,
	\begin{align*}
		&\textstyle\P\left\{ \sup_{t\in[t_j,t_{j+1}]}|X(t)-X(t_j)| 
			\ge C_2 (t_{j+1}-t_j)^\alpha + z \right\}\\
		&\le 2\exp\left(-\frac{z^2}{2\sup_{t\in[t_j,t_{j+1}]}
			\E(|X(t)-X(t_j)|^2)} \right).
%			\qquad\forall z>0.
	\end{align*}
	This and \eqref{Dudley}
	implicitly yield a constant $C_3$ --  independent of $(j\,,z)$ -- such that
	$\sup_{t\in[t_j,t_{j+1}]}\E(|X(t)-X(t_j)|^2) \le 
	C_3 (t_{j+1}-t_j)^{2\alpha}$
	and hence, for every $z>0$:
	\[
		\P\left\{ \sup_{t\in[t_j\,,t_{j+1}]}|X(t)-X(t_j)| 
		\ge C_2 (t_{j+1}-t_j)^\alpha + z \right\}
		\le 2\exp\left(-\frac{z^2}{2C_3(t_{j+1}-t_j)^{2\alpha}} \right).
	\]
	In particular, if $j$ is large enough to ensure that 
	$t_{j+1}-t_j\le (\varepsilon/(2C_2))^{1/\alpha} t_j$, then
	\begin{align*}
		&\textstyle\P\left\{
			\sup_{t\in[t_j,t_{j+1}]} |X(t)-X(t_j)| \ge \varepsilon t_j^\alpha\right\}\\
		&\textstyle\le\P\left\{ \sup_{t\in[t_j,t_{j+1}]}|X(t)-X(t_j)| 
			\ge C_2 (t_{j+1}-t_j)^\alpha + (\varepsilon/2)t_j^\alpha \right\}\\
		&\le 2\exp\left(-\frac{\varepsilon^2}{8C_3}
			\left|\frac{t_j}{t_{j+1}-t_j}\right|^{2\alpha} \right).
	\end{align*}
	We can now appeal to \eqref{T1} and the above, using a union bound,
	in order to deduce the lemma from the facts that $t_{j+1}-t_j=\mathcal{o}(t_j)$
	and $m(n)\to\infty$ as $j,n\to\infty$.
\end{proof}

One can apply Lemma \ref{lem:discrete} to produce a number of examples
of sequences $\{t_j\}_{j=1}^\infty$ along which,
\begin{equation}\label{eq:lim}
	\P\left\{ |X(t_j)|\le  ct_j^\alpha\ \forall j=1,\ldots,n\right\}
	= t_n^{-\lambda(c)+\mathcal{o}(1)}\qquad\text{as $n\to\infty$},
\end{equation}
for the same $\lambda$ as in Theorem \ref{th:main}. Recall that $X$ denotes a
$\fBm$ throughout this discussion.

\begin{example}
	For a simple example consider $t_n=n$ and $m(n)=\lfloor n^\delta\rfloor$
	where $\delta\in(0\,,1)$. Lemma \ref{lem:discrete} and Theorem \ref{th:main}
	together yield the following: For every $\varepsilon,c>0$,
	\begin{align*}
		&\P\left\{ |X(j)|\le  cj^\alpha\ \forall j=1,\ldots,n\right\}
			\le n^{-(1-\delta)\lambda(c+\varepsilon)+\mathcal{o}(1)}
			+ (K+\mathcal{o}(1))
			\sum_{n^\delta\le j\le n} \e^{-\varepsilon^2j^{2\alpha}/K}\\
		&\hskip0.6in
			=	n^{-(1-\delta)\lambda(c+\varepsilon)+\mathcal{o}(1)}
			+ \exp\left(-\frac{(\varepsilon^2+\mathcal{o}(1))n^{2\alpha\delta}}{K}
			\right)=	n^{-(1-\delta)\lambda(c+\varepsilon)+\mathcal{o}(1)},
	\end{align*}
	as $n\to\infty$. In light of \eqref{LB} and Theorem \ref{th:main},
	the preceding shows that the sequence defined by $t_j=j$
	satisfies \eqref{eq:lim}.
\end{example}

\begin{example}
	Open Problem \ref{OP2} asks whether \eqref{eq:lim} holds
	for the sequence $t_n=\exp(n)$. Among other things, 
	this example will imply that
	if $\alpha\approx1$ then \eqref{eq:lim} can hold for 
	sequences whose growth rate fall barely below $\exp(n^{2/3})$.
	Consider $q\in(0\,,1)$, and define
	$t_n = \exp ( n^q)$ for every $n\in\N$. Since
	$\frac{t_j}{t_{j+1}-t_j} =\frac{j^{1-q}}{q +\mathcal{o}(1)}$
	as $j\to\infty$,
	it follows from Lemma \ref{lem:discrete}
	and Theorem \ref{th:main} that, as long as $m(n)\to\infty$ strictly monotonically
	and $m(n)=\mathcal{o}(n)$ when $n\gg1$,
	\begin{align*}
		&\P\left\{ |X(t_j)|\le  ct_j^\alpha\ \forall j=1,\ldots,n\right\}\\
		&\le \e^{-(\lambda(c+\varepsilon) + \mathcal{o}(1))
			\left[n^q-m(n)^q\right]}
			+ \sum_{j=m(n)}^n \exp\left(-
			\frac{(\varepsilon^2+\mathcal{o}(1))j^{2(1-q)\alpha}}{K} \right)\\
		&\le \e^{-(\lambda(c+\varepsilon) + \mathcal{o}(1))
			\left[n^q-m(n)^q\right]}
			+ \exp\left(-
			\frac{(\varepsilon^2+\mathcal{o}(1))m(n)^{2(1-q)\alpha}}{K} \right).
	\end{align*}
	We want the first term on the right-hand side to dominate the second; that is,
	$n^q-m(n)^q < m(n)^{2(1-q)\alpha}$ and $m(n)\ll n.$
	Try $m(n)=n^p$ where $0<p<1$ in order to see that this happens if and only if
	$q < \frac{2\alpha p}{1+2\alpha p}.$
	Since $p$ can be as close to one as we would like,
	it follows that if $q < 2\alpha/(1+2\alpha)$, then \eqref{eq:lim}
	is satisfied for the sequence defined by $t_n=\exp(n^q)$.
\end{example}

\begin{example}
	For our third and final example we consider the sequence defined by
	$t_n = \exp( |\log n|^q),$
	where $q>1$ is fixed. Since
	$\frac{t_{j+1}-t_j}{t_j} = \frac{(q+\mathcal{o}(1))|\log j|^{q-1}}{j}$
	as $j\to\infty$,
	we may apply Lemma \ref{lem:discrete} with $m(n)=n^\delta$, where $\delta>0$
	is fixed but as small as we like, in order to see that the sequence $t_n=\exp( |\log n|^q)$
	satisfies \eqref{eq:lim} when $q>1$.
\end{example}

\bigskip

\noindent
{\bf Acknowledgments.} 
The authors thank Professor Guangqu Zheng, and two anonymous
referees, for making several helpful 
remarks that led to improvements in the paper.

\bibliography{fBmBC}

\providecommand{\bysame}{\leavevmode\hbox to3em{\hrulefill}\thinspace}
\providecommand{\MR}{\relax\ifhmode\unskip\space\fi MR }
% \MRhref is called by the amsart/book/proc definition of \MR.
\providecommand{\MRhref}[2]{%
  \href{http://www.ams.org/mathscinet-getitem?mr=#1}{#2}
}
\providecommand{\href}[2]{#2}
\begin{thebibliography}{100}

\bibitem{AliliPatie}
L.~Alili and P.~Patie, \emph{Boundary crossing identities for {B}rownian motion
  and some nonlinear {ODE}'s}, Proc. Amer. Math. Soc. \textbf{142} (2014),
  no.~11, 3811--3824. \MR{3251722}

\bibitem{Anderson55}
T.~W. Anderson, \emph{The integral of a symmetric unimodal function over a
  symmetric convex set and some probability inequalities}, Proc. Amer. Math.
  Soc. \textbf{6} (1955), 170--176. \MR{69229}

\bibitem{BarczyDoring}
M\'aty\'as Barczy and Leif D\"oring, \emph{On entire moments of self-similar
  {M}arkov processes}, Stoch. Anal. Appl. \textbf{31} (2013), no.~2, 191--198.
  \MR{3021485}

\bibitem{BassBurdzy96}
Richard~F. Bass and Krzysztof Burdzy, \emph{A critical case for {B}rownian slow
  points}, Probab. Theory Related Fields \textbf{105} (1996), no.~1, 85--108.
  \MR{1389733}

\bibitem{Berman1978}
Simeon~M. Berman, \emph{Gaussian processes with biconvex covariances}, J.
  Multivariate Anal. \textbf{8} (1978), no.~1, 30--44. \MR{517591}

\bibitem{Berman1987}
\bysame, \emph{Spectral conditions for local nondeterminism}, Stochastic
  Process. Appl. \textbf{27} (1987), no.~1, 73--84. \MR{934530}

\bibitem{Bertoin}
Jean Bertoin, \emph{{L}\'evy {P}rocesses}, Cambridge Tracts in Mathematics,
  vol. 121, Cambridge University Press, Cambridge, 1996. \MR{1406564}

\bibitem{BertoinKorchemski}
Jean Bertoin and Igor Kortchemski, \emph{Self-similar scaling limits of
  {M}arkov chains on the positive integers}, Ann. Appl. Probab. \textbf{26}
  (2016), no.~4, 2556--2595. \MR{3543905}

\bibitem{BertoinYor2002A}
Jean Bertoin and Marc Yor, \emph{The entrance laws of self-similar {M}arkov
  processes and exponential functionals of {L}\'evy processes}, Potential Anal.
  \textbf{17} (2002), no.~4, 389--400. \MR{1918243}

\bibitem{BertoinYor2002B}
\bysame, \emph{On the entire moments of self-similar {M}arkov processes and
  exponential functionals of {L}\'evy processes}, Ann. Fac. Sci. Toulouse Math.
  (6) \textbf{11} (2002), no.~1, 33--45. \MR{1986381}

\bibitem{BlackwellFreedman}
David Blackwell and David Freedman, \emph{A remark on the coin tossing game},
  Ann. Math. Statist. \textbf{35} (1964), 1345--1347. \MR{169257}

\bibitem{Borell74}
Christer Borell, \emph{Convex measures on locally convex spaces}, Ark. Mat.
  \textbf{12} (1974), 239--252. \MR{388475}

\bibitem{Borell}
\bysame, \emph{The {B}runn-{M}inkowski inequality in {G}auss space}, Invent.
  Math. \textbf{30} (1975), no.~2, 207--216. \MR{399402}

\bibitem{Breiman}
Leo Breiman, \emph{First exit times from a square root boundary}, Proc. {F}ifth
  {B}erkeley {S}ympos. {M}ath. {S}tatist. and {P}robability ({B}erkeley,
  {C}alif., 1965/66), {V}ol. {II}: {C}ontributions to {P}robability {T}heory,
  {P}art 2, Univ. California Press, Berkeley, CA, 1967, pp.~9--16. \MR{212865}

\bibitem{CaballeroChaumont}
M.~E. Caballero and L.~Chaumont, \emph{Weak convergence of positive
  self-similar {M}arkov processes and overshoots of {L}\'evy processes}, Ann.
  Probab. \textbf{34} (2006), no.~3, 1012--1034. \MR{2243877}

\bibitem{CKP}
L.~Chaumont, A.~E. Kyprianou, and J.~C. Pardo, \emph{Some explicit identities
  associated with positive self-similar {M}arkov processes}, Stochastic
  Process. Appl. \textbf{119} (2009), no.~3, 980--1000. \MR{2499867}

\bibitem{CKPR}
Lo\"ic Chaumont, Andreas Kyprianou, Juan~Carlos Pardo, and V\'ictor Rivero,
  \emph{Fluctuation theory and exit systems for positive self-similar {M}arkov
  processes}, Ann. Probab. \textbf{40} (2012), no.~1, 245--279. \MR{2917773}

\bibitem{CP}
Loic Chaumont and J.~C. Pardo, \emph{The lower envelope of positive
  self-similar {M}arkov processes}, Electron. J. Probab. \textbf{11} (2006),
  no. 49, 1321--1341. \MR{2268546}

\bibitem{CT}
Y.~S. Chow and H.~Teicher, \emph{On second moments of stopping rules}, Ann.
  Math. Statist. \textbf{37} (1966), 388--392. \MR{192528}

\bibitem{Coeurjolly}
Jean-Fran\c{c}ois Coeurjolly, \emph{Hurst exponent estimation of locally
  self-similar {G}aussian processes using sample quantiles}, Ann. Statist.
  \textbf{36} (2008), no.~3, 1404--1434. \MR{2418662}

\bibitem{CDG}
Michele Coti~Zelati, Theodore~D. Drivas, and Rishabh~S. Gvalani, \emph{Mixing
  by statistically self-similar {G}aussian random fields}, J. Stat. Phys.
  \textbf{191} (2024), no.~5, Paper No. 61, 11. \MR{4746649}

\bibitem{CuzickDupreez}
Jack Cuzick and Johannes~P. DuPreez, \emph{Joint continuity of {G}aussian local
  times}, Ann. Probab. \textbf{10} (1982), no.~3, 810--817. \MR{659550}

\bibitem{Dalang}
Robert~C. Dalang, \emph{Extending the martingale measure stochastic integral
  with applications to spatially homogeneous s.p.d.e.'s}, Electron. J. Probab.
  \textbf{4} (1999), no. 6, 29. \MR{1684157}

\bibitem{DLMX21}
Robert~C. Dalang, Cheuk~Yin Lee, Carl Mueller, and Yimin Xiao, \emph{Multiple
  points of {G}aussian random fields}, Electron. J. Probab. \textbf{26} (2021),
  Paper No. 17, 25. \MR{4235468}

\bibitem{Davis76}
Burgess Davis, \emph{On the {$L\sp{p}$} norms of stochastic integrals and other
  martingales}, Duke Math. J. \textbf{43} (1976), no.~4, 697--704. \MR{418219}

\bibitem{Davis83}
\bysame, \emph{On {B}rownian slow points}, Z. Wahrsch. Verw. Gebiete
  \textbf{64} (1983), no.~3, 359--367. \MR{716492}

\bibitem{DavisPerkins85}
Burgess Davis and Edwin Perkins, \emph{Brownian slow points: the critical
  case}, Ann. Probab. \textbf{13} (1985), no.~3, 779--803. \MR{799422}

\bibitem{DT}
Krzysztof D\c{e}bicki and Kamil Tabi\'s, \emph{Pickands-{P}iterbarg constants
  for self-similar {G}aussian processes}, Probab. Math. Statist. \textbf{40}
  (2020), no.~2, 297--315. \MR{4206416}

\bibitem{DDK}
Steffen Dereich, Leif D\"oring, and Andreas~E. Kyprianou, \emph{Real
  self-similar processes started from the origin}, Ann. Probab. \textbf{45}
  (2017), no.~3, 1952--2003. \MR{3650419}

\bibitem{Dobrushin}
R.~L. Dobrushin, \emph{Gaussian and their subordinated self-similar random
  generalized fields}, Ann. Probab. \textbf{7} (1979), no.~1, 1--28.
  \MR{515810}

\bibitem{Doring}
Leif D\"oring, \emph{A jump-type {SDE} approach to real-valued self-similar
  {M}arkov processes}, Trans. Amer. Math. Soc. \textbf{367} (2015), no.~11,
  7797--7836. \MR{3391900}

\bibitem{DoringBarczy}
Leif D\"oring and M\'aty\'as Barczy, \emph{A jump type {SDE} approach to
  positive self-similar {M}arkov processes}, Electron. J. Probab. \textbf{17}
  (2012), no. 94, 39. \MR{2994842}

\bibitem{Dudley1967}
R.~M. Dudley, \emph{The sizes of compact subsets of {H}ilbert space and
  continuity of {G}aussian processes}, J. Functional Analysis \textbf{1}
  (1967), 290--330. \MR{220340}

\bibitem{Dudley}
\bysame, \emph{Random linear functionals}, Trans. Amer. Math. Soc. \textbf{136}
  (1969), 1--24. \MR{264726}

\bibitem{DFL}
R.~M. Dudley, Jacob Feldman, and L.~Le~Cam, \emph{On seminorms and
  probabilities, and abstract {W}iener spaces}, Ann. of Math. (2) \textbf{93}
  (1971), 390--408. \MR{279272}

\bibitem{Fekete}
M.~Fekete, \emph{\"{U}ber die {V}erteilung der {W}urzeln bei gewissen
  algebraischen {G}leichungen mit ganzzahligen {K}oeffizienten}, Math. Z.
  \textbf{17} (1923), no.~1, 228--249. \MR{1544613}

\bibitem{Fernique}
Xavier Fernique, \emph{Int\'egrabilit\'e{} des vecteurs gaussiens}, C. R. Acad.
  Sci. Paris S\'er. A-B \textbf{270} (1970), A1698--A1699. \MR{266263}

\bibitem{FKM}
Mohammud Foondun, Davar Khoshnevisan, and Pejman Mahboubi, \emph{Analysis of
  the gradient of the solution to a stochastic heat equation via fractional
  {B}rownian motion}, Stoch. Partial Differ. Equ. Anal. Comput. \textbf{3}
  (2015), no.~2, 133--158. \MR{3350450}

\bibitem{Funaki}
T.~Funaki, \emph{Singular limit for reaction-diffusion equation with
  self-similar {G}aussian noise}, New trends in stochastic analysis
  ({C}haringworth, 1994), World Sci. Publ., River Edge, NJ, 1997, pp.~132--152.
  \MR{1654352}

\bibitem{GnedinPitman}
A.~Gnedin and J.~Pitman, \emph{Self-similar and {M}arkov composition
  structures}, Zap. Nauchn. Sem. S.-Peterburg. Otdel. Mat. Inst. Steklov.
  (POMI) \textbf{326} (2005), 59--84, 280--281. \MR{2183216}

\bibitem{GoodmanKuelbs}
Victor Goodman and James Kuelbs, \emph{Rates of clustering for some {G}aussian
  self-similar processes}, Probab. Theory Related Fields \textbf{88} (1991),
  no.~1, 47--75. \MR{1094077}

\bibitem{GV}
S.~E. Graversen and J.~Vuolle-Apiala, \emph{{$\alpha$}-self-similar {M}arkov
  processes}, Probab. Theory Relat. Fields \textbf{71} (1986), no.~1, 149--158.
  \MR{814666}

\bibitem{GreenwoodPerkinsB}
P.~Greenwood and E.~Perkins, \emph{Limit theorems for excursions from a moving
  boundary}, Teor. Veroyatnost. i Primenen. \textbf{29} (1984), no.~4,
  703--714. \MR{773439}

\bibitem{GreenwoodPerkinsA}
Priscilla Greenwood and Edwin Perkins, \emph{A conditioned limit theorem for
  random walk and {B}rownian local time on square root boundaries}, Ann.
  Probab. \textbf{11} (1983), no.~2, 227--261. \MR{690126}

\bibitem{Gross}
Leonard Gross, \emph{Abstract {W}iener spaces}, Proc. {F}ifth {B}erkeley
  {S}ympos. {M}ath. {S}tatist. and {P}robability ({B}erkeley, {C}alif.,
  1965/66), {V}ol. {II}: {C}ontributions to {P}robability {T}heory, {P}art 1,
  Univ. California Press, Berkeley, CA, 1967, pp.~31--42. \MR{212152}

\bibitem{GHN}
Qingsong Gu, Jiaxin Hu, and Sze-Man Ngai, \emph{Geometry of self-similar
  measures on intervals with overlaps and applications to sub-{G}aussian heat
  kernel estimates}, Commun. Pure Appl. Anal. \textbf{19} (2020), no.~2,
  641--676. \MR{4043760}

\bibitem{GundySiegmund}
Richard~F. Gundy and David Siegmund, \emph{On a stopping rule and the central
  limit theorem}, Ann. Math. Statist. \textbf{38} (1967), 1915--1917.
  \MR{217842}

\bibitem{HP}
Martin Hairer and \'Etienne Pardoux, \emph{A {W}ong-{Z}akai theorem for
  stochastic {PDE}s}, J. Math. Soc. Japan \textbf{67} (2015), no.~4,
  1551--1604. \MR{3417505}

\bibitem{HN2017}
Daniel Harnett and David Nualart, \emph{Decomposition and limit theorems for a
  class of self-similar {G}aussian processes}, Stochastic analysis and related
  topics, Progr. Probab., vol.~72, Birkh\"auser/Springer, Cham, 2017,
  pp.~99--116. \MR{3737626}

\bibitem{HN2018}
\bysame, \emph{Central limit theorem for functionals of a generalized
  self-similar {G}aussian process}, Stochastic Process. Appl. \textbf{128}
  (2018), no.~2, 404--425. \MR{3739502}

\bibitem{HHX}
Minhao Hong, Heguang Liu, and Fangjun Xu, \emph{Limit theorems for additive
  functionals of some self-similar {G}aussian processes}, Ann. Appl. Probab.
  \textbf{34} (2024), no.~6, 5462--5497. \MR{4839629}

\bibitem{KalbasiMountford}
Kamran Kalbasi and Thomas Mountford, \emph{On the probability distribution of
  the local times of diagonally operator-self-similar {G}aussian fields with
  stationary increments}, Bernoulli \textbf{26} (2020), no.~2, 1504--1534.
  \MR{4058376}

\bibitem{KentWood}
John~T. Kent and Andrew T.~A. Wood, \emph{Estimating the fractal dimension of a
  locally self-similar {G}aussian process by using increments}, J. Roy.
  Statist. Soc. Ser. B \textbf{59} (1997), no.~3, 679--699. \MR{1452033}

\bibitem{KSXZ}
Davar Khoshnevisan, Jason Swanson, Yimin Xiao, and Liang Zhang, \emph{Weak
  existence of a solution to a differential equation driven by a very rough
  fbm}, 2013, Unpublished manuscript available at
  \url{https://arxiv.org/abs/1309.3613}.

\bibitem{KimSongVondracek}
Panki Kim, Renming Song, and Zoran Vondra\v{c}ek, \emph{Positive self-similar
  {M}arkov processes obtained by resurrection}, Stochastic Process. Appl.
  \textbf{156} (2023), 379--420. \MR{4520527}

\bibitem{Kolmogorov}
A.~N. Kolmogoroff, \emph{Wienersche {S}piralen und einige andere interessante
  {K}urven im {H}ilbertschen {R}aum}, C. R. (Doklady) Acad. Sci. URSS (N.S.)
  \textbf{26} (1940), 115--118. \MR{3441}

\bibitem{KMR}
A.~E. Kyprianou, M.~Motala, and V.~Rivero, \emph{Williams' path decomposition
  for self-similar {M}arkov processes in {$\Bbb R^d$}}, Electron. J. Probab.
  \textbf{29} (2024), Paper No. 132, 31. \MR{4801600}

\bibitem{Kyprianou}
Andreas~E. Kyprianou, \emph{{I}ntroductory {L}ectures on {F}luctuations of
  {L}\'evy {P}rocesses with {A}pplications}, Universitext, Springer-Verlag,
  Berlin, 2006. \MR{2250061}

\bibitem{Lai}
Dejian Lai, \emph{Group sequential tests under fractional {B}rownian motion in
  monitoring clinical trials}, Stat. Methods Appl. \textbf{19} (2010), no.~2,
  277--286. \MR{2651453}

\bibitem{LaiETAL}
Dejian Lai, Lemuel~A. Moy\'e, Barry~R. Davis, Lisa~E. Brown, and Frank~M.
  Sacks, \emph{Brownian motion and long-term clinical trial recruitment}, J.
  Statist. Plann. Inference \textbf{93} (2001), no.~1-2, 239--246. \MR{1822399}

\bibitem{Lamperti62}
John Lamperti, \emph{Semi-stable stochastic processes}, Trans. Amer. Math. Soc.
  \textbf{104} (1962), 62--78. \MR{138128}

\bibitem{Lamperti67}
\bysame, \emph{Continuous state branching processes}, Bull. Amer. Math. Soc.
  \textbf{73} (1967), 382--386. \MR{208685}

\bibitem{Lamperti72}
\bysame, \emph{Semi-stable {M}arkov processes. {I}}, Z.
  Wahrscheinlichkeitstheorie und Verw. Gebiete \textbf{22} (1972), 205--225.
  \MR{307358}

\bibitem{LatalaMatlak}
Rafa\l{} Lata\l{}a and Dariusz Matlak, \emph{Royen's proof of the {G}aussian
  correlation inequality}, {G}eometric {A}spects of {F}unctional {A}nalysis,
  Lecture Notes in Math., vol. 2169, Springer, Cham, 2017, pp.~265--275.
  \MR{3645127}

\bibitem{Ledoux}
Michel Ledoux, \emph{{I}soperimetry and {G}aussian {A}nalysis}, Lectures on
  probability theory and statistics ({S}aint-{F}lour, 1994), Lecture Notes in
  Math., vol. 1648, Springer, Berlin, 1996, pp.~165--294. \MR{1600888}

\bibitem{LX23}
Cheuk~Yin Lee and Yimin Xiao, \emph{Chung-type law of the iterated logarithm
  and exact moduli of continuity for a class of anisotropic {G}aussian random
  fields}, Bernoulli \textbf{29} (2023), no.~1, 523--550. \MR{4497257}

\bibitem{LeiNualart}
Pedro Lei and David Nualart, \emph{A decomposition of the bifractional
  {B}rownian motion and some applications}, Statist. Probab. Lett. \textbf{79}
  (2009), no.~5, 619--624. \MR{2499385}

\bibitem{LingWang}
Ming Liao and Longmin Wang, \emph{Isotropic self-similar {M}arkov processes},
  Stochastic Process. Appl. \textbf{121} (2011), no.~9, 2064--2071.
  \MR{2819241}

\bibitem{LPHFMVM}
Martin Lysy, Natesh~S. Pillai, David~B. Hill, M.~Gregory Forest, John W.~R.
  Mellnik, Paula~A. Vasquez, and Scott~A. McKinley, \emph{Model comparison and
  assessment for single particle tracking in biological fluids}, J. Amer.
  Statist. Assoc. \textbf{111} (2016), no.~516, 1413--1426. \MR{3601698}

\bibitem{MM}
Vitalii Makogin and Yuliya Mishura, \emph{Example of a {G}aussian self-similar
  field with stationary rectangular increments that is not a fractional
  {B}rownian sheet}, Stoch. Anal. Appl. \textbf{33} (2015), no.~3, 413--428.
  \MR{3339311}

\bibitem{Mandelbrot}
Benoit~B. Mandelbrot, \emph{The {F}ractal {G}eometry of {N}ature}, W. H.
  Freeman and Co., San Francisco, CA, 1982. \MR{665254}

\bibitem{MandelbrotVanNess}
Benoit~B. Mandelbrot and John~W. Van~Ness, \emph{Fractional {B}rownian motions,
  fractional noises and applications}, SIAM Rev. \textbf{10} (1968), 422--437.
  \MR{242239}

\bibitem{MasonXiao}
J.~D. Mason and Yimin Xiao, \emph{Sample path properties of
  operator-self-similar {G}aussian random fields}, Teor. Veroyatnost. i
  Primenen. \textbf{46} (2001), no.~1, 94--116. \MR{1968707}

\bibitem{MS}
Muneya Matsui and Narn-Rueih Shieh, \emph{The {L}amperti transforms of
  self-similar {G}aussian processes and their exponentials}, Stoch. Models
  \textbf{30} (2014), no.~1, 68--98. \MR{3175862}

\bibitem{Mattila}
Pertti Mattila, \emph{{G}eometry of {S}ets and {M}easures in {E}uclidean
  {S}paces}, Cambridge Studies in Advanced Mathematics, vol.~44, Cambridge
  University Press, Cambridge, 1995, Fractals and rectifiability. \MR{1333890}

\bibitem{MicloPatieSarkar}
Laurent Miclo, Pierre Patie, and Rohan Sarkar, \emph{Discrete self-similar and
  ergodic {M}arkov chains}, Ann. Probab. \textbf{50} (2022), no.~6, 2085--2132.
  \MR{4499275}

\bibitem{Minlos}
R.~A. Minlos, \emph{Generalized random processes and their extension in
  measure}, Trudy Moskov. Mat. Ob\v s\v c. \textbf{8} (1959), 497--518.
  \MR{108851}

\bibitem{Molchan}
G.~M. Molchan, \emph{Burgers equation with self-similar {G}aussian initial
  data: tail probabilities}, J. Statist. Phys. \textbf{88} (1997), no.~5-6,
  1139--1150. \MR{1478064}

\bibitem{MonradPitt}
Ditlev Monrad and Loren~D. Pitt, \emph{Local nondeterminism and {H}ausdorff
  dimension}, Seminar on stochastic processes, 1986 ({C}harlottesville, {V}a.,
  1986), Progr. Probab. Statist., vol.~13, Birkh\"auser Boston, Boston, MA,
  1987, pp.~163--189. \MR{902433}

\bibitem{MWB}
Carl Mueller and Zhixin Wu, \emph{A connection between the stochastic heat
  equation and fractional {B}rownian motion, and a simple proof of a result of
  {T}alagrand}, Electron. Commun. Probab. \textbf{14} (2009), 55--65.
  \MR{2481666}

\bibitem{Muraoka}
Hiroshi Muraoka, \emph{Self-similar isotropic {G}aussian random fields},
  Probability {T}heory and {M}athematical {S}tatistics ({K}iev, 1991), World
  Sci. Publ., River Edge, NJ, 1992, pp.~210--214. \MR{1212126}

\bibitem{Novikov}
A.~A. Novikov, \emph{The stopping times of a {W}iener process}, Teor.
  Verojatnost. i Primenen. \textbf{16} (1971), 458--465. \MR{315798}

\bibitem{NX}
David Nualart and Fangjun Xu, \emph{Asymptotic behavior for an additive
  functional of two independent self-similar {G}aussian processes}, Stochastic
  Process. Appl. \textbf{129} (2019), no.~10, 3981--4008. \MR{3997669}

\bibitem{Pardo}
J.~C. Pardo, \emph{The upper envelope of positive self-similar {M}arkov
  processes}, J. Theoret. Probab. \textbf{22} (2009), no.~2, 514--542.
  \MR{2501333}

\bibitem{PardoRivero}
Juan~Carlos Pardo and V\'ictor Rivero, \emph{Self-similar {M}arkov processes},
  Bol. Soc. Mat. Mexicana (3) \textbf{19} (2013), no.~2, 201--235. \MR{3183994}

\bibitem{Patie}
P.~Patie, \emph{Law of the absorption time of some positive self-similar
  {M}arkov processes}, Ann. Probab. \textbf{40} (2012), no.~2, 765--787.
  \MR{2952091}

\bibitem{Perkins83}
Edwin Perkins, \emph{On the {H}ausdorff dimension of the {B}rownian slow
  points}, Z. Wahrsch. Verw. Gebiete \textbf{64} (1983), no.~3, 369--399.
  \MR{716493}

\bibitem{Pitt78}
Loren~D. Pitt, \emph{Local times for {G}aussian vector fields}, Indiana Univ.
  Math. J. \textbf{27} (1978), no.~2, 309--330. \MR{471055}

\bibitem{Rivero2003}
V\'ictor Rivero, \emph{A law of iterated logarithm for increasing self-similar
  {M}arkov processes}, Stoch. Stoch. Rep. \textbf{75} (2003), no.~6, 443--472.
  \MR{2029617}

\bibitem{Rivero2005}
\bysame, \emph{Recurrent extensions of self-similar {M}arkov processes and
  {C}ram\'er's condition}, Bernoulli \textbf{11} (2005), no.~3, 471--509.
  \MR{2146891}

\bibitem{Rivero2007}
\bysame, \emph{Recurrent extensions of self-similar {M}arkov processes and
  {C}ram\'er's condition. {II}}, Bernoulli \textbf{13} (2007), no.~4,
  1053--1070. \MR{2364226}

\bibitem{Royen}
Thomas Royen, \emph{A simple proof of the {G}aussian correlation conjecture
  extended to some multivariate gamma distributions}, Far East J. Theor. Stat.
  \textbf{48} (2014), no.~2, 139--145. \MR{3289621}

\bibitem{Sato}
Ken-iti Sato, \emph{{L}\'evy {P}rocesses and {I}nfinitely {D}ivisible
  {D}istributions}, revised ed., Cambridge Studies in Advanced Mathematics,
  vol.~68, Cambridge University Press, Cambridge, 2013, Translated from the
  1990 Japanese original. \MR{3185174}

\bibitem{Shepp}
L.~A. Shepp, \emph{A first passage problem for the {W}iener process}, Ann.
  Math. Statist. \textbf{38} (1967), 1912--1914. \MR{217879}

\bibitem{SM}
Naftali~R. Smith and Satya~N. Majumdar, \emph{Condensation transition in large
  deviations of self-similar {G}aussian processes with stochastic resetting},
  J. Stat. Mech. Theory Exp. (2022), no.~5, Paper No. 053212, 31. \MR{4473216}

\bibitem{Song}
Xiaoming Song, \emph{Large deviations for functionals of some self-similar
  {G}aussian processes}, Stochastics \textbf{93} (2021), no.~3, 311--336.
  \MR{4229257}

\bibitem{Talagrand}
Michel Talagrand, \emph{Regularity of {G}aussian processes}, Acta Math.
  \textbf{159} (1987), no.~1-2, 99--149. \MR{906527}

\bibitem{T95}
\bysame, \emph{Hausdorff measure of trajectories of multiparameter fractional
  {B}rownian motion}, Ann. Probab. \textbf{23} (1995), no.~2, 767--775.
  \MR{1334170}

\bibitem{TeicherZhang}
Henry Teicher and Cun-Hui Zhang, \emph{Moments of some stopping rules}, J.
  London Math. Soc. (2) \textbf{57} (1998), no.~2, 503--512. \MR{1644253}

\bibitem{VLCVM}
Hippolyte Verdier, Fran\c{c}ois Laurent, Alhassan Cass\'e, Christian~L.
  Vestergaard, and Jean-Baptiste Masson, \emph{Variational inference of
  fractional {B}rownian motion with linear computational complexity}, Phys.
  Rev. E \textbf{106} (2022), no.~5, Paper No. 055311, 12. \MR{4527264}

\bibitem{Verzani95}
John Verzani, \emph{Slow points in the support of historical {B}rownian
  motion}, Ann. Probab. \textbf{23} (1995), no.~1, 56--70. \MR{1330760}

\bibitem{Vidmar}
Matija Vidmar, \emph{Exit problems for positive self-similar {M}arkov processes
  with one-sided jumps}, S\'eminaire de {P}robabilit\'es {LI}, Lecture Notes in
  Math., vol. 2301, Springer, Cham, [2022] \copyright 2022, pp.~91--115.
  \MR{4461023}

\bibitem{Walsh}
John~B. Walsh, \emph{An {I}ntroduction to {S}tochastic {P}artial {D}ifferential
  {E}quations}, \'Ecole d'\'et\'e{} de probabilit\'es de {S}aint-{F}lour,
  {XIV}---1984, Lecture Notes in Math., vol. 1180, Springer, Berlin, 1986,
  pp.~265--439. \MR{876085}

\bibitem{Xiao1996}
Yimin Xiao, \emph{Hausdorff measure of the sample paths of {G}aussian random
  fields}, Osaka J. Math. \textbf{33} (1996), no.~4, 895--913. \MR{1435460}

\bibitem{Xiao}
\bysame, \emph{Asymptotic results for self-similar {M}arkov processes},
  Asymptotic {M}ethods in {P}robability and {S}tatistics ({O}ttawa, {ON},
  1997), North-Holland, Amsterdam, 1998, pp.~323--340. \MR{1661490}

\bibitem{ZhangLai}
Qiang Zhang and Dejian Lai, \emph{Fractional {B}rownian motion and long term
  clinical trial recruitment}, J. Statist. Plann. Inference \textbf{141}
  (2011), no.~5, 1783--1788. \MR{2763208}

\bibitem{ZPMPGR}
L.~Zunino, D.~G. P\'erez, M.~T. Mart\'in, A.~Plastino, M.~Garavaglia, and O.~A.
  Rosso, \emph{Characterization of {G}aussian self-similar stochastic processes
  using wavelet-based informational tools}, Phys. Rev. E (3) \textbf{75}
  (2007), no.~2, 021115, 10. \MR{2354015}

\end{thebibliography}
\bibliographystyle{amsplain}

\end{document}